\documentclass{amsart}

\usepackage{amsmath}
\usepackage{amsthm}
\usepackage{amsfonts}
\usepackage{amssymb}
\usepackage{mathabx}

\usepackage{pgf,tikz}
\usepackage{mathrsfs}
\usetikzlibrary{arrows}
\usepackage[mathscr]{euscript}

\usepackage{graphics}
\usepackage{epsfig}
\usepackage{color}
\usepackage{verbatim}

\usepackage[utf8]{inputenc}
\usepackage{enumerate}
\usepackage[bookmarks = true, colorlinks=true, linkcolor = black, citecolor = black, menucolor = black, urlcolor = black]{hyperref}

\newcommand\R{{\mathbb{R}}}

\newcommand\Z{{\mathbb{Z}}}
\newcommand\N{{\mathbb{N}}}

\newcommand\bi{{\textbf{i}}}
\newcommand\bj{{\textbf{j}}}

\newcommand\supp{\operatorname{supp}}

\newcommand\dist{\operatorname{dist}}

\newcommand\CA{{\mathcal A}}

\newcommand\CN{{\mathcal N}}

\newcommand\CS{{\mathcal S}}

\theoremstyle{plain}
  \newtheorem{theorem}{Theorem}

  \newtheorem{proposition}[theorem]{Proposition}
  
  \newtheorem{lemma}[theorem]{Lemma}
  \newtheorem{corollary}[theorem]{Corollary}

\theoremstyle{definition}

\date{\today}
\title[Control of pseudodifferential operators by maximal functions]{Control of pseudodifferential operators by maximal functions via weighted inequalities}
\author{David Beltran}
\address{School of Mathematics, University of Birmingham, Edgbaston, Birmingham, B15 2TT, UK}
\email{dbeltran89@gmail.com}
\thanks{This work was supported by the European Research Council [grant
number 307617]}

\begin{document}

\begin{abstract}
We establish general weighted $L^2$ inequalities for pseudodifferential operators associated to the Hörmander symbol classes $S^m_{\rho,\delta}$. Such inequalities allow to control these operators by fractional ``non-tangential" maximal functions, and subsume the optimal range of Lebesgue space bounds for pseudodifferential operators. As a corollary, several known Muckenhoupt type bounds are recovered, and new bounds for weights lying in the intersection of the Muckenhoupt and reverse Hölder classes are obtained. The proof relies on a subdyadic decomposition of the frequency space, together with applications of the Cotlar--Stein almost orthogonality principle and a quantitative version of the symbolic calculus.
\end{abstract}
\maketitle

\section{Introduction}\label{sec:Intro}

Given a smooth function $a \in C^\infty(\R^d \times \R^d)$, we define the associated pseudodifferential operator $T_a$ by
$$
T_af(x)=\int_{\mathbb{R}^d}e^{ix\cdot\xi}a(x,\xi)\widehat{f}(\xi)d\xi,
$$
where $f \in \mathcal{S}$ and $\widehat{f}$ denotes the Fourier transform of $f$. The smooth function $a$ is typically referred to as the \textit{symbol}, and throughout the rest of the paper, we will assume that $a$ belongs to the symbol classes $S^m_{\rho,\delta}$ introduced by Hörmander in \cite{Hormander}; recall that $S^m_{\rho,\delta}$ consists of all $a \in C^\infty(\R^d \times \R^d)$ satisfying the differential inequalities
\begin{equation}\label{symbol}
|\partial_x^\nu \partial_\xi^\sigma a(x,\xi)| \lesssim (1+|\xi|)^{m-\rho|\sigma|+\delta|\nu|}
\end{equation}
for all multi-indices $\nu,\sigma \in \N^d$, where $m \in \R$ and $0\leq \delta, \rho \leq 1$.\footnote{We use the notation $A \lesssim B$ to denote that there is a constant $C$ such that $A \leq CB$. The implicit constant will always be independent of the weight $w$. We omit the constant factors of $\pi$ coming from our normalisation of the Fourier transform.}

The study of pseudodifferential operators was initiated by Kohn and Nirenberg \cite{KN} and Hörmander \cite{Hormander}, and it has played a central role in the theory of partial differential equations. The $L^p$-boundedness of these operators has been extensively studied, see for instance the work of Calderón and Vaillancourt \cite{CV1972} for the $L^2$-boundedness of the classes $S^0_{\rho,\rho}$, with $0 \leq \rho<1$, or Hörmander \cite{Hormander}, Fefferman \cite{Fe73} or Stein \cite{bigStein} for $L^p$ bounds for the symbol classes $S_{\rho,\delta}^m$. Weighted $L^p$-boundedness in the context of the Muckenhoupt $A_p$ classes has also been studied, see for example the work of Miller \cite{Miller}, Chanillo and Torchinsky \cite{CT}, or the most recent work of Michalowski, Rule and Staubach \cite{MRS10, MRScan}. 

In this paper we are interested in establishing weighted inequalities for $T_a$ valid for any weight $w$, and that fall well beyond the classical $A_p$ theory. Our main result is the following.

\begin{theorem}\label{MainThm}
Let $a \in S^{m}_{\rho,\delta}$, where $m \in \R$, $0\leq \delta \leq \rho \leq 1$, $\delta <1$. Then
\begin{equation}\label{WeightedIneqSymbols}
\int_{\R^d} |T_a f|^2 w \lesssim \int_{\R^d} |f|^2 M^2\mathscr{M}_{\rho,m}M^5w
\end{equation}
holds for any non-negative $w \in L^1_{\text{loc}}(\R^d)$, where
$$
\mathscr{M}_{\rho,m} w(x):= \sup_{(y,r) \in \Lambda_\rho(x)} \frac{1}{|B(y,r)|^{1+2m/d}} \int_{B(y,r)} w
$$
and
$$
\Lambda_\rho(x):=\{(y,r)\in \R^{d} \times (0,1) : |y-x| \leq r^{\rho}\}.
$$
\end{theorem}

Our theorem covers the full range of admissible values for $m, \rho, \delta$ except for an endpoint case corresponding to the symbol classes $S_{1,1}^m$. This is to be expected, as it is well known that there are symbols in the class $S^0_{1,1}$ that fail to be bounded on $L^2$, and thus \eqref{WeightedIneqSymbols} would fail on taking $w \equiv 1$.

The maximal operator $\mathscr{M}_{\rho,m}$ may be interpreted as a fractional Hardy--Littlewood maximal function associated with the approach region $\Lambda_{\rho}$, which is a ``truncated" standard cone in the classical case $\rho=1$, and allows a certain order of tangential approach when $0<\rho<1$. In this latter case, the maximal functions $\mathscr{M}_{\rho,m}$ are closely related to those considered by Nagel and Stein \cite{NS} in a different context. 
%
%
Also, these maximal functions are best possible in terms of Lebesgue space bounds. An elementary duality argument and an application of Hölder's inequality reveal that if $a \in S^{m}_{\rho,\delta}$, where $m \in \R$, $0\leq \delta \leq \rho \leq 1$, $\delta <1$, the inequality \eqref{WeightedIneqSymbols} implies
$$
\|T_a\|_{p \to q} \lesssim \|\mathscr{M}_{\rho,m}\|_{(q/2)' \to (p/2)'}^{1/2}
$$
for $p,q \geq 2$. This allows us to transfer $L^p-L^q$ bounds for $\mathscr{M}_{\rho,m}$ to bounds for $T_a$. The bounds satisfied by the maximal operator $\mathscr{M}_{\rho,m}$, which were established in \cite{BB}, allow one to recover the optimal bounds for the symbol classes $S_{\rho,\delta}^m$. In particular, we may reprove that $T_a$ is bounded on $L^p$, for $2\leq p<\infty$, under the condition $m \leq -d(1-\rho)|1/p-1/2|$; see Stein \cite{bigStein} for a review of the bounds on the symbol classes.


These maximal functions are also significant improvements of some variants of the Hardy--Littlewood maximal function. In particular, a crude application of Hölder's inequality in the definition of $\mathscr{M}_{\rho,m}$ reveals the pointwise estimate
\begin{equation}\label{MaximalDomination}
\mathscr{M}_{\rho,m} w \leq (Mw^s)^{1/s}
\end{equation}
when $2sm=(\rho-1)d$, for any $s \geq 1$. On the level of Lebesgue space bounds, the maximal operators $\mathscr{M}_{\rho,m}$ are bounded on $L^s$, for $s>1$, when $2sm=(\rho-1)d$, a property that the maximal functions $(Mw^s)^{1/s}$ do not enjoy. These observations allow us to reconcile Theorem \ref{MainThm} with existing results in the context of weighted $A_p$ theory. In particular, for $s=1$ one may obtain the following.

\begin{corollary}\label{M8}
Let $a \in S^{-d(1-\rho)/2}_{\rho,\delta}$, where $0 \leq \delta \leq \rho \leq 1$, $\delta <1$. Then
\begin{equation}\label{CorWeighted}
\int_{\R^d} |T_a f|^2 w \lesssim \int_{\R^d} |f|^2 M^{8}w.
\end{equation}
\end{corollary}

The inequalities \eqref{CorWeighted} improve on the existing two-weighted inequalities with controlling maximal function $(Mw^s)^{1/s}$, which are implicit in the works \cite{CT,MRS10} from the elementary observation that $(M w^s)^{1/s} \in A_1$ for any $s>1$. We remark that in the case of the standard symbol class $S^0:=S^0_{1,0}$ and the classes $S^0_{1,\delta}$, with $\delta<1$, the inequality \eqref{CorWeighted} holds with maximal operator $M^3$; this is a consequence of a result of Pérez \cite{Pe94} for Calderón--Zygmund operators. We note that the number of compositions of $M$ in \eqref{WeightedIneqSymbols} and \eqref{CorWeighted} is unlikely to be sharp here and we do not concern ourselves with such finer points in this paper.

If $w \in A_1$, Corollary \ref{M8} immediately yields that $T_a$ is bounded on $L^2(w)$ for $w \in A_1$ and $a \in S^{-d(1-\rho)/2}_{\rho,\delta}$, $0 \leq \delta \leq \rho \leq 1$, $\delta <1$. By standard Rubio de Francia extrapolation theory \cite{GCRdF}, one may recover the best known Muckenhoupt-type weighted estimates for such symbol classes, previously obtained by Chanillo and Torchinksy \cite{CT} in the case $\delta < \rho$, and by Michalowski, Rule and Staubach \cite{MRS10} for $\delta \geq \rho$.

\begin{corollary}[\cite{CT, MRS10}]\label{CT and MRS}
Let $a \in S^{-d(1-\rho)/2}_{\rho,\delta}$, where $0 \leq \delta \leq \rho \leq 1$, $\delta <1$. Then $T_a$ is bounded on $L^p(w)$ for $w \in A_{p/2}$ and $2\leq p < \infty$.
\end{corollary}

Similarly, one may also deduce weighted results of Muckenhoupt-type for the symbol classes $S^m_{\rho,\delta}$ with $-d(1-\rho)/2<m<0$. In view of the pointwise estimate \eqref{MaximalDomination}, one has that $T_a$ is bounded on $L^2(w)$ provided $w^s \in A_1$, where $s=(\rho-1)d/2m$. This may be understood in the context of weights lying in the intersection of the Muckenhoupt and reverse Hölder classes,\footnote{Recall that $w \in RH_s$ if there exists a constant $C$ such that $\big(\frac{1}{|B|} \int_B w^s \big)^{1/s} \leq C \frac{1}{|B|} \int_B w$ for any ball $B$ in $\R^d$.} as $w^s \in A_1$ if and only if $w \in A_1 \cap RH_s$. By an extrapolation theorem of Auscher and Martell \cite{AM}, one quickly deduces the following.

\begin{corollary}
Let $a \in S^{m}_{\rho,\delta}$, where $-d(1-\rho)/2 \leq m <0$ and $0 \leq \delta \leq \rho < 1$. Let $s=(\rho-1)d/2m$. Then $T_a$ is bounded on $L^p(w)$ for all $w \in A_{p/2} \cap RH_{(2s'/p)'}$ and $2 \leq p < 2s'$.
\end{corollary}

Of course this contains Corollary \ref{CT and MRS} when $m=-d(1-\rho)/2$. The estimates for the range $-d(1-\rho)/2<m<0$ appear to be new to the best of our knowledge. It is interesting to compare them with \cite[Theorem 3.11]{MRS10}.


As a contextual remark, we should add that our study of weighted inequalities of the type \eqref{WeightedIneqSymbols} is motivated by a question raised by Stein in the late 1970s. In \cite{Stein79}, he suggested the possibility that the disc multiplier might be controlled via a general weighted $L^2$ inequality by some variant of the universal maximal function
$$
\CN w(x):= \sup_{T \in x} \frac{1}{|T|} \int_T w;
$$
here the supremum is taken over all rectangles $T$ containing the point $x$. This conjecture is still very much open for $d\geq 2$, although positive results were obtained in the case of radial weights by Carbery, Romera and Soria \cite{CRS}. A similar question was also raised by Córdoba \cite{CordKak} in the more general context of the Bochner--Riesz multipliers. For further progress concerning this conjecture we refer the interested reader to the work of Carbery \cite{CarWeight}, Christ \cite{Christ}, Carbery and Seeger \cite{CS}, Bennett, Carbery, Soria and Vargas \cite{BCSV} or Lee, Rogers and Seeger \cite{LRSw}; see also Córdoba and Rogers \cite{CR2014} for a weighted inequality in a related oscillatory context.

In a more classical context, such types of weighted inequalities were first studied by Fefferman and Stein \cite{FS} for the Hardy--Littlewood maximal function $M$, with the controlling maximal function being $M$ itself. In the framework of Calderón--Zygmund operators, this question was addressed by Córdoba and Fefferman \cite{CF}, Wilson \cite{Wi89} and Pérez \cite{Pe94}, where the controlling maximal function is a minor variant of $M$. A similar result was recently proved by the author \cite{Bel2015} for the Carleson operator, combining some of the ideas developed by Pérez in \cite{Pe95} with the recent developments connecting Calderón--Zygmund theory with sparse operators; see for instance the work of Lerner \cite{LeA2,LeNew}, Di Plinio and Lerner \cite{DiPL} or Lacey \cite{Lac2015}. 


\textit{Structure of the paper.} The paper is organised as follows. In Section \ref{sec:Multiplier} we discuss our strategy for attacking the problem. Section \ref{sec:Preliminaries} contains some preliminary results to which we will appeal in the proof of Theorem \ref{MainThm}. In Section \ref{sec:ProofMain} we give a detailed proof of our main result, which relies on auxiliary results proved in Sections \ref{sec:origin} and \ref{sec:ProofDyadicPiece}. Some appendices are included at the end of this paper for completeness. 

\textit{Acknowledgements.} The author would like to thank his supervisor Jon Bennett for stimulating conversations and valuable comments on the exposition of this paper. He also thanks Alessio Martini for an interesting discussion.

\section{Proof strategy}\label{sec:Multiplier}

A precedent for Theorem \ref{MainThm} is the work of Bennett and the author \cite{BB}, who established weighted inequalities for certain classes of Fourier multipliers. Among many things, they showed that given $0 \leq \alpha \leq 1$ and $\beta \in \R$, if $m:\R^d \to \mathbb{C}$ is a function supported in $\{\xi \in \R^d : |\xi| \geq 1\}$ satisfying the differential inequalities
\begin{equation}\label{mult}
|D^\sigma m(\xi)|\lesssim |\xi|^{-\beta-(1-\alpha)|\sigma|}
\end{equation}
for all multi-indices $\sigma \in \N^d$ such that $|\sigma| \leq \lfloor \frac{d}{2} \rfloor +1$, the operator $T_m$ associated to the Fourier multiplier $m$ satisfies the weighted inequality \eqref{WeightedIneqSymbols}.\footnote{The results in \cite{BB} hold for any value of $\alpha \in \R$, although this will not be relevant in this paper. Indeed, the weighted inequalities there follow from stronger \textit{pointwise} results.}

In order to establish that result, the concept of \textit{subdyadic} balls proved to be crucial. We say that a euclidean ball $B$ in $\mathbb{R}^d$ is $\alpha$-\textit{subdyadic} if $\dist(B,0)\geq 1$ and
$$r(B) \sim \dist(B,0)^{1-\alpha},$$
where $r(B)$ denotes the radius of $B$. Observe that for $\alpha>0$, typically $r(B)\ll\dist(B,0)$, making it natural to refer to such balls as subdyadic. The prototypical example of a decomposition of $\{\xi \in \R^d : |\xi|\geq 1\}$ into $\alpha$-subdyadic balls is to decompose the space into dyadic annuli $A_k:=\{\xi \in \R^d: |\xi| \sim 2^k\}$ and cover each $A_k$ by a family of $O(2^{\alpha dk})$ balls of radius $O(2^{k(1-\alpha)})$ with bounded overlap.\footnote{Observe that the case $\alpha=0$ reduces to the case of a standard decomposition of $\R^d$ into dyadic balls.} This two-stage decomposition example is implicitly used in the theory of pseudodifferential operators, as it may be extracted from Stein \cite{bigStein}. We refer to \cite{BB} for a more detailed description of the subdyadic decomposition, which has its roots in the sequence of papers \cite{BCSV,BH,Ben2014}.

The reason to decompose the frequency space into subdyadic balls is that the multipliers $m$ satisfying the differential inequalities \eqref{mult} are effectively constant on such balls. A manifestation of that principle is that it is possible to prove, with rather elementary techniques, that if $B$ is an $\alpha$-subdyadic ball and $f_B$ is a function Fourier supported in $B$ (or in a slightly enlargement of $B$) then
\begin{equation}\label{WeightedLocalised}
\int_{\R^d} |T_m f_B|^2 w \lesssim \dist(B,0)^{-2\beta} \int_{\R^d} |f_B|^2 Mw,
\end{equation}
with constant independent of the ball $B$. Of course, this corresponds to an analogue of Theorem \ref{MainThm} over the class of functions $f$ with that specific Fourier support.

The idea then is to use suitable Littlewood--Paley type decompositions to reduce the proof of the weighted estimates \eqref{WeightedIneqSymbols} to the estimate \eqref{WeightedLocalised}. This comes in two stages, with the first one being a reduction, via a standard Littlewood--Paley decomposition, to functions $f_k$ whose Fourier transforms are supported in the dyadic annuli $A_k$.

The reduction from $f_k$ to the functions $f_B$ is more subtle. Observe that if $B$, $B'$ are $\alpha$-subdyadic balls lying on the same dyadic annulus $A_k$, then $r(B)\sim r(B') \sim 2^{k(1-\alpha)}$. Now, imagine the following situation. Let $f_B$, $f_{B'}$ be functions whose Fourier support lies in $B$ and $B'$ respectively. Let $\widetilde{w}$ be a weight function with Fourier support lying in a ball centered at the origin of radius $r(B)\sim r(B') \sim 2^{k(1-\alpha)}$. Then, Parseval's theorem reveals the orthogonality property
$$
\int_{\R^d} f_B \overline{f_{B'}} \widetilde{w} = \int_{\R^d} \widehat{f_B} \widehat{\overline{f_{B'}}} \ast \widehat{\widetilde{w}}= 0
$$
if $\dist(B,B') \gtrsim r(B)$. As $T_m$ is a convolution operator, the frequency variables of $f$ and $T_mf$ are the same, and this orthogonality remains valid for $T_m f_B$ and $T_m f_{B'}$. This means that only ``diagonal" terms contribute to the whole sum, that is,
$$
\int_{\R^d} |\sum_{r(B) \sim 2^{k(1-\alpha)}}T_mf_B|^2 \widetilde{w}=\int_{\R^d} \sum_{r(B) \sim r(B') \sim 2^{k(1-\alpha)} } T_m f_B \overline{T_m f_{B'}} \widetilde{w} \sim \int_{\R^d} \sum_{r(B) \sim 2^{k(1-\alpha)} } |T_m f_B|^2 \widetilde{w},
$$
and for such diagonal terms, one may invoke the easier inequality \eqref{WeightedLocalised}.


In view of the above discussion, the similarity of the differential inequalities \eqref{symbol} satisfied by a symbol $a$ to those satisfied by the multipliers $m$ (see \eqref{mult}) suggests a decomposition of the $\xi$-space, where $\xi$ corresponds to the frequency variable of $f$, into $(1-\rho)$-subdyadic balls. However, $T_a$ is a non-translation-invariant operator. This adds complexity with respect to the Fourier multiplier case, and more delicate arguments seem to be required. In particular, observe that the frequency variables of $f$ and $T_af$ are not the same; this was a key property satisfied by Fourier multiplier operators in the above analysis.

The key idea for the pseudodifferential operators case will be that despite $T_a f_B$ and $T_a f_{B'}$ not being orthogonal with respect to the weight $\widetilde{w}$, the contribution of the ``off-diagonal terms" to the term 
$$
\int_{\R^d} |\sum_{r(B) \sim 2^{k(1-\alpha)}}T_af_B|^2 \widetilde{w}
$$
is very small if $\dist(B,B')\gtrsim r(B)$. This may be seen as a certain almost orthogonality property between the pieces $T_{a_B}$ and $T_{a_{B'}}$. To exploit this, it will be appropriate to make use of the Cotlar--Stein almost orthogonality principle.


Of course, the previous ideas require us to find a suitable weight $\widetilde{w}$ controlling an essentially arbitrary weight $w$. The weight $\widetilde{w}$, as described above, will only be effective to obtain almost orthogonality properties between $T_{a}f_B$ and $T_{a} f_{B'}$ if $B$ and $B'$ are subdyadic balls satisfying $r(B) \sim r(B')$, that is, if $B$ and $B'$ lie on the same dyadic annulus. Thus, after a dyadic Littlewood--Paley type reduction, we control $w$ by a suitably band-limited weight $\widetilde{w_k}$ in each dyadic annulus $A_k$. The weights $\widetilde{w_k}$ satisfy the above properties, and taking the supremum over all $k \geq 0$ will give raise to the maximal functions $\mathscr{M}_{\rho,m}$. Due to the non-translation-invariant nature of $T_a$, Littlewood--Paley theory will not suffice for our purposes, and a quantitative version of the \textit{symbolic calculus} will also be needed.

Observe that a $(1-\rho)$-subdyadic decomposition is only suitable in $\{\xi \in \R^d : |\xi| \geq 1\}$. This may be easily overcome, as a symbol $a(x,\xi)$ satisfying the differential inequalities \eqref{symbol} behaves differently in the regions $\{|\xi| \leq 1\}$ and $\{|\xi|\geq 1\}$. Observe that the differential inequalities \eqref{symbol} on $\{|\xi| \leq 1\}$ become
$$
|\partial_x^\nu \partial_\xi^\sigma a(x,\xi)| \lesssim 1
$$
for all multi-indices $\nu, \sigma \in \N^d$. When $a$ is supported in $\{|\xi| \leq 1\}$, these will suffice to deduce an appropriate two-weighted inequality for $T_{a}$ by elementary means.

\section{Preliminaries}\label{sec:Preliminaries}

The aim of this section is to provide the reader with some standard results, or minor variants of them, to which we shall appeal in the proof of Theorem \ref{MainThm}.

\subsection{Weighted Littlewood--Paley theory}

Here we present forward and reverse weighted $L^2$ inequalities for a dyadic square function of Littlewood--Paley type. We will use this to reduce the proof of Theorem \ref{MainThm} to the class of functions whose Fourier support lies in a dyadic annulus.

Let $P:\R^d \to \R$ be a smooth function such that $\supp(\widehat{P}) \subseteq \{\xi \in \R^d: 3/4 \leq |\xi| \leq 3\}$. For any $k \in \Z$, let $P_k$ be defined by $\widehat{P}_k(\xi)=\widehat{P}(2^{-k}\xi)$ and let $\Delta_k$ be the operator given by 
$$
\widehat{\Delta_k f}(\xi)= \widehat{P}_k(\xi)\widehat{f}(\xi).
$$
The following forward estimate follows from a more general result of Wilson \cite{Wi07}; see also \cite{BH,Ben2014} for similar formulations of the statement.

\begin{proposition}[\cite{Wi07}]\label{DyadicRecoupling}
$$
\int_{\R^d} \sum_{k \in \Z} |\Delta_k f|^2 w \lesssim \int_{\R^d} |f|^2 Mw.
$$
\end{proposition}

The reverse estimate is slightly less standard, and corresponds to a $d$-dimensional version of a result in \cite{BH}. Here we assume that the functions $\{\widehat{P}_k\}_{k \in \Z}$ define a partition of unity, that is 
$$
\sum_{k \in \Z} \widehat{P}(2^{-k}\xi)=1.
$$

\begin{proposition}[\cite{BH}]\label{DyadicDecoupling}
$$
\int_{\R^d} |f|^2 w \lesssim \int_{\R^d} \sum_{k \in \Z} |\Delta_kf|^2  M^3w.
$$
\end{proposition}

\subsection{Composition of a frequency cut-off with a pseudodifferential operator}\label{subsec:symbcalc}

The composition structure of pseudodifferential operators has been extensively studied; we refer to the work of Hörmander \cite{Hormander} in the case of the symbol classes $S^m_{\rho,\delta}$. We require the following quantitative version when the outermost symbol is a cut-off function on the frequency space adapted to a dyadic annulus.

\begin{theorem}\label{RefinedSymbolicCalculus}
Let $\varphi \in \CS$ be such that $\supp(\widehat{\varphi}) \subseteq \{|\xi|\sim 1\}$ and given $R>1$, let $\varphi_R$ be defined by $\widehat{\varphi}_R(\xi):=\widehat{\varphi}(R^{-1}\xi)$. Let $a \in S_{\rho,\delta}^{m}$, where $0 \leq \delta \leq \rho$ and $\delta < 1$. Then, there exists a symbol $c \in S^{m}_{\rho,\delta}$ such that
$$
T_c=T_{\widehat{\varphi}_R} \circ T_a.
$$
Moreover, for $\epsilon \geq 0$ and $\kappa >0$, the symbol
$$
e^N:=c-\sum_{|\gamma|<N} \frac{i^{-|\gamma|}}{\gamma!} \partial_\xi^\gamma \widehat{\varphi}_R \partial_x^\gamma a \in S^{m-N(1-\delta)+d\delta+\kappa \delta+\epsilon}_{\rho,\delta}
$$
for all $N>\frac{d\delta +\kappa \delta + \epsilon}{1-\delta}$, and satisfies
\begin{equation}\label{ErrorDiffIneq}
|\partial_x^\nu \partial_\xi^\sigma e^N(x,\xi)|\lesssim R^{-\epsilon} (1+|\xi|)^{m-N(1-\delta)+d\delta+\kappa \delta+\epsilon -|\sigma|\rho + |\nu|\delta}
\end{equation}
for any multi-indices $\nu,\sigma \in \N^d$. 
\end{theorem}

This very specific version of the more general symbolic calculus in \cite{Hormander} allows us to obtain quantitative control for the differential inequalities satisfied by the error term $e^N$ in terms of $R$, which corresponds to the scale of the frequency projection $\varphi_R$. The implicit constants in \eqref{ErrorDiffIneq} depend on finitely many $C^k$ norms of $\widehat{\varphi}$ and on the implicit constants in the differential inequalities \eqref{symbol} satisfied by $a$, and they will be acceptable for our purposes for being independent of the parameter $R$.

We remark that the order of the error symbol $e^N$ in Theorem \ref{RefinedSymbolicCalculus} is not necessarily sharp here, but one may choose $N$ sufficiently large so that $e^N$ has sufficiently large negative order. Modulo such an error term, we may understand the composition of $\varphi_R$ with a pseudodifferential operator as the action of the pseudodifferential operator itself, and some other pseudodifferential operators of lower order, on functions with frequency support on the dyadic annulus $\{ |\xi|\sim R\}$. We provide the proof of Theorem \ref{RefinedSymbolicCalculus} in Appendix \ref{app:SymbCalc} for completeness, which consists of a careful modification of the symbolic calculus developed in \cite{bigStein} for the standard symbol classes $S^m$.

%

\subsection{The kernel of a pseudodifferential operator}\label{subsec:kernel}

%
%
%

A pseudodifferential operator with symbol of sufficiently negative order is to all intents and purposes a convolution operator with an integrable kernel. This is an easy consequence of the following observation in Hörmander \cite{Hormander}. Let $a \in S^m_{\rho,\delta}$, $m \in \R, 0\leq \delta, \rho \leq 1$, $\delta<1$ and let $K(x,y)$ denote the distribution kernel of $T_a$. Then if $\gamma \in \N^d$ satisfies $m-|\gamma|\rho<-d$, the distribution $(x-y)^\gamma K(x,y)$ coincides with a function,
\begin{equation}\label{proofKernel}
(x-y)^\gamma K(x,y) = \int_{\R^d} e^{i(x-y)\cdot \xi} (-iD_\xi)^\gamma a(x,\xi) d\xi.
\end{equation}
In view of the differential inequalities \eqref{symbol}, this quickly allows us to deduce that if a symbol $a \in S^m_{\rho,\delta}$ has sufficiently negative order, that is, $m <-d$, then
\begin{equation*}
|K(x,y)|\lesssim \frac{1}{(1+|x-y|^2)^{L/2}}
\end{equation*}
for any $L\geq 0$. In particular, taking $L>d$, one may control the pseudodifferential operator $T_a$ by a convolution operator with an integrable kernel.

This elementary observation will be very useful to handle the pseudodifferential operator associated with the error symbol $e^N$ obtained after an application of Theorem \ref{RefinedSymbolicCalculus}. In view of the differential inequalities \eqref{ErrorDiffIneq} satisfied by $e^N$, the identity \eqref{proofKernel} reveals that if $N$ is chosen such that $m-N(1-\delta)+d\delta+\kappa \delta+\epsilon <-d$ then the kernel $K_{e^N}$ associated to the symbol $e^N$ satisfies
\begin{equation}\label{ErrorKernel}
|K_{e^N}(x,y)|\lesssim \frac{R^{-\epsilon}}{(1+|x-y|^2)^{L/2}}
\end{equation}
for any $L\geq 0$. As in \eqref{ErrorDiffIneq}, the implicit constant here is independent of $R$, and only depends on finitely many $C^k$ norms of $\widehat{\varphi}$ and on the implicit constants in the differential inequalities \eqref{symbol} satisfied by $a$. Taking $L>d$, this allows us to bound $T_{e^N}$ by an integrable convolution kernel with a quantitative control of the constant in terms of the scale of the frequency projection $\varphi_R$. As we shall see in Section \ref{sec:ProofMain}, such quantitative control is required for summability purposes in the proof of Theorem \ref{MainThm}.

%
\subsection{Almost orthogonality}

The proof of Theorem \ref{TheoremDyadicPiece} relies on an application of the Cotlar--Stein almost orthogonality principle.

\begin{lemma}[Cotlar--Stein, \cite{bigStein} p. 280]\label{Cotlar--Stein}
Let $\{T_j\}_{j \in \Z^d}$ be a family of operators and $T=\sum_{j \in \Z^d} T_j$. Let $\{c(j)\}_{j \in \Z^d}$ be a family of positive constants such that
$$
A=\sum_{j \in \Z^d} c(j) < \infty
$$
and assume that
$$
\|T_i^*T_j\|_{2 \to 2} \leq c(i-j)^2,
$$
$$
\|T_iT_j^*\|_{2 \to 2} \leq c(i-j)^2.
$$
Then 
$$
\|T\|_{2 \to 2} \leq A.
$$
\end{lemma}

\subsection{$L^2$-boundedness of integral operators}

We also require the following standard version of the Schur test, which is a simple consequence of the Cauchy--Schwarz inequality; see for example Theorem 5.2 in \cite{Halmos}.

\begin{lemma}[Schur's test, \cite{Halmos}]\label{SchurLemma}
Suppose $T$ is given by
$$
Tf(x)=\int_{\R^d}K(x,z)f(z)dz
$$
and assume there exist measurable functions $h_1, h_2 >0$ and positive constants $C_1$ and $C_2$ such that
\begin{equation*}
\int_{\R^d} |K(x,z)|h_1(z)dz \leq C_1 h_2(x) \;\;\;\; \text{ and }  \;\;\;\; \int_{\R^d} |K(x,z)|h_2(x)dx \leq C_2 h_1(z).
\end{equation*}
Then
$$
\|T\|_{2 \to 2} \leq (C_1 C_2)^{1/2}.
$$
\end{lemma}

\section{Proof of Theorem \ref{MainThm}}\label{sec:ProofMain}

Let $a \in S^m_{\rho,\delta}$ with $m \in \R$, $0 \leq \delta \leq \rho \leq 1$, $\delta <1$. By the embeddings of the symbol classes is enough to prove Theorem \ref{MainThm} for $a \in S^m_{\rho,\rho}$ with $0 \leq\rho<1$, and $a \in S^m_{1,\delta}$ with $\delta < 1$; recall that
$$
S_{\rho_1,\delta_1}^{m_1} \subseteq S_{\rho_2,\delta_2}^{m_2} \:\:\:\: \text{if} \:\:\:\: m_1 \leq m_2, \:\: \rho_1 \geq \rho_2, \:\: \delta_1 \leq \delta_2. 
$$
Observe that the upcoming Theorem \ref{TheoremDyadicPiece} is also valid for the symbol classes $S^m_{1,\delta}$ with $\delta <1$, as they are embedded in $S_{1,1}^m$.

As discussed in Section \ref{sec:Multiplier}, a symbol $a$ satisfying the differential inequalities \eqref{symbol} behaves differently in the regions $\{|\xi| \leq 1\}$ and $\{|\xi| \geq 1\}$. Let $\eta \in C^\infty(\R^d)$ be a smooth function supported in $|\xi| \leq 2$ and let $a_0(x,\xi)=a(x,\xi)\eta(\xi)$ and $a_1$ be such that $a=a_0+a_1$. Theorem \ref{MainThm} will follow from establishing the required weighted inequalities for both $T_{a_0}$ and $T_{a_1}$.

In view of \eqref{symbol}, the symbol $a_0$ satisfies the differential inequalities
$$
|\partial_x^\nu \partial_\xi^\sigma a_0(x,\xi)| \lesssim 1
$$
for all multi-indices $\nu ,\sigma \in \N^d$. This, together with the support condition on the variable $\xi$ that we just imposed in $a_0(x,\xi)$, leads to the following rather elementary weighted inequality.

\begin{proposition}\label{WeightedOrigin}
$$
\int_{\R^d} |T_{a_0}f|^2 w \lesssim \int_{\R^d} |f|^2 A_1^*w.
$$
where $A_1^*w:=\sup_{r \geq 1} A_r w$ and
$$
A_r w (x):= \frac{1}{|B(x,r)|}\int_{B(x,r)}w.
$$
\end{proposition}

We provide a proof of this proposition in Section \ref{sec:origin}. The inequality \eqref{WeightedIneqSymbols} for $T_{a_0}$ follows from noting that
$$
A_1^*w \lesssim A_1 A_1^* w \lesssim \mathscr{M}_{\rho,m} A_1^* w \lesssim \mathscr{M}_{\rho,m} Mw \lesssim M^2\mathscr{M}_{\rho,m}M^5w.
$$

The difficulty relies thus on understanding the operator $T_{a_1}$. We will reduce the proof of Theorem \ref{MainThm} to the following theorem, which corresponds to an analogous statement but over the class of functions whose Fourier support lies in a dyadic annulus.

\begin{theorem}\label{TheoremDyadicPiece}
Let $a \in S^{m}_{\rho,\rho}$, where $0\leq \rho \leq 1$. Let $f$ be a function such that $\supp(\widehat{f}) \subseteq \{\xi \in \R^d: |\xi| \sim R\}$, where $R\geq 1$. Then
$$
\int_{\R^d} |T_a f|^2 w \lesssim \int_{\R^d} |f|^2  \mathcal{A}_{\rho,m,R}w
$$
uniformly in $R\geq 1$, where 
$$
\mathcal{A}_{\rho,m,R}w(x):=R^{2m} \int_{\R^d} \Big(\sup_{|y-z|\leq R^{-\rho}}w(z) \Big) \frac{R^{\rho d}}{(1+R^{2\rho}|x-y|^2)^{N_0/2}}dy
$$
and $N_0$ is any natural number satisfying $N_0>d$.
\end{theorem}

We postpone to Section \ref{sec:ProofDyadicPiece} the proof of this theorem, which involves the two-stage decomposition briefly described in Section \ref{sec:Multiplier}.

The reduction to Theorem \ref{TheoremDyadicPiece} is done as follows. A first application of Proposition \ref{DyadicDecoupling} to the function $T_{a_1}f$ gives
$$
\int_{\R^d} |T_{a_1} f|^2 w \lesssim \sum_{k \geq 0} \int_{\R^d} |\Delta_k (T_{a_1} f)|^2 M^3w.
$$
Let $\Phi$ be a smooth function such that $\widehat{\Phi}=1$ in $\{\eta \in \R^d : |\eta| \leq 3\}$ and define $\Phi_k$ by $\widehat{\Phi}_k(\eta)=\widehat{\Phi}(2^{-k}\eta)$ for any $k \geq 0$. As $\widehat{\Delta_k g}(\eta)= \widehat{P}(2^{-k}\eta)\widehat{g}(\eta)$ and $\supp(\widehat{P}) \subseteq \{\eta \in \R^d: 3/4 \leq |\eta| \leq 3\}$, we have $\Delta_k (T_{a_1} f)= \Delta_k (T_{a_1} f) \ast \Phi_k$. An application of the Cauchy--Schwarz inequality and Fubini's theorem gives
\begin{equation}\label{smoothWeight}
\int_{\R^d} |\Delta_k (T_{a_1} f)|^2 M^3 w = \int_{\R^d} |\Delta_k (T_{a_1} f) \ast \Phi_k|^2 M^3 w \lesssim \int_{\R^d} |\Delta_k (T_{a_1} f) |^2 |\Phi_k| \ast M^3w,
\end{equation}
uniformly in $k \geq 0$, as the functions $\Phi_k$ are normalised on $L^1(\R^d)$.

At this stage, one would like to interchange $\Delta_k$ and $T_{a_1}$ in order to apply Theorem \ref{TheoremDyadicPiece}. As discussed in Section \ref{subsec:symbcalc}, this may be done provided we introduce terms of lower order.  As $\delta <1$, fixing $\epsilon>0$ and $\kappa>0$, an application of Theorem \ref{RefinedSymbolicCalculus} for any $k \geq 0$ gives 
$$
\Delta_k(T_{a_1} f) = T_{a_1}(\Delta_k f) + \sum_{1 \leq |\gamma|<N} \frac{i^{-|\gamma|}}{\gamma!} T^\gamma_k f + T_{e_k} f,
$$
where 
$$
T^\gamma_k f(x):= \int_{\R^d} e^{i x \cdot \xi} \partial_\xi^\gamma \widehat{P}_k (\xi) \partial_x^\gamma a_1 (x,\xi) \widehat{f}(\xi),
$$
and $e_k$ is a symbol satisfying
$$
|\partial_x^\nu \partial_\xi^\sigma e_k(x,\xi)|\lesssim 2^{-k \epsilon} (1+|\xi|)^{m-N(1-\delta)+d\delta+\kappa\delta+\epsilon -|\sigma|\rho + |\nu|\delta}
$$
for any multi-indices $\nu, \sigma \in \N^d$. Here $\gamma \in \N^d$, and we choose $N$ to be a positive integer satisfying 
$$
m-N(1-\delta)+d\delta+\kappa\delta+\epsilon <-d;
$$
for ease of notation we removed the dependence of $N$ in the error term $e_k$, as $N$ is a chosen fixed number independent of $k$. Such a choice of $N$ allows one to argue as in Section \ref{subsec:kernel}, and the inequality \eqref{ErrorKernel} reads here as
$$
|K_{e_k}(x,y)|\lesssim \frac{2^{-k \epsilon}}{(1+|x-y|^2)^{L/2}}
$$
for any $L \geq 0$. Taking $L>d$, and setting $\Psi^{(L)}(x):=(1+|x|^2)^{-L/2}$, an application of the Cauchy--Schwarz inequality and Fubini's theorem gives
$$
\int_{\R^d}|T_{e_k} f|^2 |\Phi_k| \ast M^3w \lesssim 2^{-2k\epsilon} \int_{\R^d} |f|^2 \Psi^{(L)} \ast |\Phi_k| \ast M^3w  \lesssim 2^{-2k\epsilon} \int_{\R^d} |f|^2 M^2 \mathscr{M}_{\rho,m} M^5 w, 
$$
with implicit constant independent of $k \geq 0$; the last inequality follows from the observation that
$$
\Psi^{(L)} \ast |\Phi_k| \ast M^3 w \lesssim A_1^* M^4 w \lesssim A_1 A_1^* M^4 w \lesssim \mathscr{M}_{\rho,m}  A_1^* M^4 w \lesssim M^2 \mathscr{M}_{\rho,m} M^5 w.
$$
This is an acceptable bound for each $T_{e_k}$, as summing over all $k \geq 0$ we obtain
$$
\sum_{k \geq 0} \int_{\R^d} |T_{e_k}f|^2 |\Phi_k| \ast M^3 w \lesssim \sum_{k \geq 0} 2^{-2k\epsilon} \int_{\R^d} |f|^2 M^2 \mathscr{M}_{\rho,m}M^5 w \lesssim \int_{\R^d} |f|^2 M^2 \mathscr{M}_{\rho,m}M^5 w
$$
for any $\varepsilon>0$.

For the term corresponding to $T_{a_1}(\Delta_k f)$, we invoke Theorem \ref{TheoremDyadicPiece},
$$
\int_{\R^d} |T_{a_1}(\Delta_k f)|^2 |\Phi_k| \ast M^3w \lesssim \int_{\R^d} |\Delta_k f|^2  \CA_{\rho,m,2^k}(|\Phi_k| \ast M^3w) \lesssim \int_{\R^d} |\Delta_ k f|^2 M \mathscr{M}_{\rho,m} M^4 w,
$$
where the last inequality follows by taking the supremum over all $k \geq 0$ in the weight function. Now, one may recouple the dyadic frequency pieces using the standard weighted Littlewood--Paley theory from Proposition \ref{DyadicRecoupling},
$$
\sum_{k \geq 0} \int_{\R^d} |\Delta_ k f|^2 M \mathscr{M}_{\rho,m} M^4 w \lesssim \int_{\R^d} |f|^2 M^2 \mathscr{M}_{\rho,m} M^5 w.
$$

Finally, we need to study the terms $T^\gamma_k$ for $1 \leq |\gamma| < N$. Observe that $\partial_\xi^\gamma \widehat{P}_k$ is supported on $\{\xi \in \R^d: 3/4 \cdot 2^k \leq |\xi| \leq 3 \cdot 2^k\}$ for any $\gamma \in \N^d$, so we are still able to use Theorem \ref{TheoremDyadicPiece} here. To this end, let $\theta$ be a smooth function such that $\widehat{\theta}(\xi) = 1 $ on $\{\xi \in \R^d: 3/4 \leq |\xi| \leq 3\}$ and that vanishes outside $\{\xi \in \R^d: 1/2 \leq |\xi| \leq 4\}$. Let $\Theta_k$ be the operator defined by $\widehat{\Theta_k g}(\xi)=\widehat{\theta}_k(\xi)\widehat{g}(\xi)$, where $\widehat{\theta}_k(\xi)= \widehat{\theta}(2^{-k}\xi)$. Then $T^\gamma_k f=T^\gamma_k(\Theta_k f)$ and observing that the symbol $\partial_\xi^\gamma \widehat{P}_k (\xi) \partial_x^\gamma a_1 (x,\xi) \in S^{m}_{\rho,\delta}$ uniformly in $k \geq 0$ (by embedding of symbol classes), Theorem \ref{TheoremDyadicPiece} leads to
$$
\int_{\R^d} |T^\gamma_k f|^2 |\Phi_k| \ast M^3w = \int_{\R^d} |T^\gamma_k (\Theta_k f)|^2 |\Phi_k| \ast M^3w   \lesssim \int_{\R^d} |\Theta_k f|^2  \CA_{\rho,m,2^k} (|\Phi_k|\ast M^3 w),
$$
uniformly in $k \geq 0$, for every $\gamma$ such that $1\leq |\gamma| < N$. The sum in $\gamma$ is not a problem as there is a finite number of terms in that sum, so
$$
\sum_{k \geq 0} \sum_{1\leq |\gamma| \leq N} \frac{1}{\gamma!} \int_{\R^d} |T^\gamma_k f|^2 |\Phi_k| \ast M^3w \lesssim \sum_{k \geq 0} \int_{\R^d} |\Theta_k f|^2  \CA_{\rho,m,2^k} (|\Phi_k|\ast M^3 w).
$$
For the sum in $k$ we use again standard weighted Littlewood--Paley theory (Proposition \ref{DyadicRecoupling}) to conclude that
$$
\sum_{k \geq 0} \int_{\R^d} |\Theta_k f|^2  \CA_{\rho,m,2^k} (|\Phi_k|\ast M^3 w) \leq \sum_{k \geq 0} \int_{\R^d} |\Theta_k f|^2 M \mathscr{M}_{\rho,m} M^5 w \lesssim \int_{\R^d} |f|^2 M^2 \mathscr{M}_{\rho,m} M^5 w,
$$
where the first inequality follows from taking the supremum in $k\geq 0$ in the weight function. Putting the pieces together, we have shown that
$$
\int_{\R^d} |T_{a_1} f|^2 w \lesssim \int_{\R^d} |f|^2 M^2 \mathscr{M}_{\rho,m} M^5 w,
$$
and therefore the proof of Theorem \ref{MainThm} is completed provided we verify the statements of Proposition \ref{WeightedOrigin} and Theorem \ref{TheoremDyadicPiece}.

\section{The part $|\xi|\leq 1$: proof of Proposition \ref{WeightedOrigin}}\label{sec:origin}

It is crucial to realise that as $a_0(x,\xi)$ has compact support in the $\xi$ variable, we may write $T_{a_0}$ as
$$
T_{a_0}f(x)=\int_{\R^d} \int_{\R^d} e^{i (x-y)\cdot \xi} a_0(x,\xi) f(y)dyd\xi,
$$
as the double integral is absolutely convergent. Denoting by $K_0$ the kernel of $T_{a_0}$,
$$
K_0(x,z)=\int_{\R^d} e^{i z \cdot \xi} a_0(x,\xi) d\xi,
$$
we may write
$$
T_{a_0} f(x)= \int_{\R^d} K_0(x,x-y) f(y)dy.
$$
We may interpret $T_{a_0}$ as the convolution of the function $K(x, \cdot)$ with $f$ evaluated at the point $x$ and
\begin{align*}
\int_{\R^d} |T_{a_0}f(x)|^2 w(x)dx & \leq \int_{\R^d} \Big(\int_{\R^d} |K_0(x,z)||f(x-z)|dz \Big)^2 w(x)dx.
\end{align*}
We split the range of integration for the inner integral in two parts, $|z|\leq 1$ and $|z|\geq 1$. For the first term, the Cauchy-Schwarz inequality, Plancherel's theorem and the estimates on $a_0$ give
\begin{align*}
\Big(\int_{|z|\leq 1} |K_0(x,z)||f(x-z)|dz \Big)^2 &  \leq  \Big(\int_{\R^d} |K_0(x,z)|^2 dz \Big) \Big( \int_{|z|\leq 1}|f(x-z)|^2dz \Big) \\
& \lesssim  \Big(\int_{|\xi|\leq 2} |a_0(x,\xi)|^2 d\xi \Big) \Big( \int_{|z|\leq 1}|f(x-z)|^2dz \Big) \\
& \lesssim \int_{\R^d} |f(x-z)|^2 \frac{1}{(1+|z|^2)^{L}}dz.
\end{align*}
Similarly, for the second term,
\begin{align*}
\Big(\int_{|z|\geq 1} |K_0(x,z)||f(x-z)|dz \Big)^2 &  \leq  \Big(\int_{\R^d} |K_0(x,z)|^2 |z|^{2\sigma} dz \Big) \Big( \int_{|z|\geq 1}\frac{1}{|z|^{2L}}|f(x-z)|^2dz \Big) \\
& \leq  \Big(\int_{\R^d} \sum_{|\sigma|=L}|z^\sigma K_0(x,z)|^2  dz \Big) \Big( \int_{|z|\geq 1}\frac{1}{|z|^{2L}}|f(x-z)|^2dz \Big) \\
& \lesssim  \Big(\int_{|\xi|\leq 2} \sum_{|\sigma|=L} |D^\sigma_\xi a_0(x,\xi)|^2 d\xi \Big) \Big( \int_{|z|\geq 1} \frac{1}{(1+|z|^2)^{L}}|f(x-z)|^2dz \Big) \\
& \lesssim \int_{\R^d} |f(x-z)|^2 \frac{1}{(1+|z|^2)^{L}}dz,
\end{align*}
where $\sigma \in \N^d$ is a multi-index of order $L$. Putting things together and setting $\Psi^{(2L)}(y)=(1+|y|^2)^{-L}$, Fubini's theorem gives
\begin{align*}
\int_{\R^d} |T_{a_0}f(x)|^2w(x)dx \lesssim \int_{\R^d} \int_{\R^d} |f(x-z)|^2 \frac{1}{(1+|z|^2)^{L}} dz w(x)dx = \int_{\R^d} |f(z)|^2 \Psi^{(2L)} \ast w (z).
\end{align*}
Proposition \ref{WeightedOrigin} follows from noting that $\Psi^{(2L)} \ast w \lesssim A_1^* w$ for $L>d/2$.

\section{The dyadic pieces in $|\xi| \geq 1$: Proof of Theorem \ref{TheoremDyadicPiece}}\label{sec:ProofDyadicPiece}

By analogy with the proof provided in \cite{bigStein} for the $L^2$-boundedness of the symbol classes $S^0_{\rho,\rho}$, with $0\leq \rho <1$, we reduce Theorem \ref{TheoremDyadicPiece} to a similar statement for the symbol classes $S_{0,0}^0$. As we shall see, this is achieved using Bessel potentials and an elementary scaling argument. For the proof of the weighted inequality for the class $S^0_{0,0}$ we perform an equally spaced decomposition and make an application of the Cotlar--Stein almost orthogonality principle.

%
%

\subsection{Reduction to the symbol classes $S_{\rho,\rho}^0$}

It is enough to prove the following version of Theorem \ref{TheoremDyadicPiece} for the symbol classes $S^0_{\rho,\rho}$.

\begin{proposition}\label{beta0}
Let $a \in S^0_{\rho,\rho}$, where $0 \leq \rho \leq 1$. Let $f$ be a function such that $\supp(\widehat{f}) \subseteq \{\xi \in \R^d : |\xi| \sim R\}$ with $R\geq 1$. Then
\begin{equation*}
\int_{\R^d} |T_a f|^2 w \lesssim \int_{\R^d} |f|^2 \mathcal{A}_{\rho,0,R}w
\end{equation*}
uniformly in $R \geq 1$.
\end{proposition}
Theorem \ref{TheoremDyadicPiece} follows from the above proposition via the following observation. Let $J_m$ denote the Bessel potential of order $m$, that is $\widehat{J_m f}(\xi)=(1+|\xi|^2)^{m/2}\widehat{f}(\xi)$. Then
$$
T_af(x)=\int_{\mathbb{R}^d}e^{ix\cdot\xi}a(x,\xi)(1+|\xi|^2)^{m/2}(1+|\xi|^2)^{-m/2}\widehat{f}(\xi)d\xi=T_{\widetilde{a}}(J_m f)(x),
$$
where $\widetilde{a}(x,\xi)=a(x,\xi)(1+|\xi|^2)^{-m/2} \in S^0_{\rho,\rho}$.  By Proposition \ref{beta0}
\begin{align*}
\int_{\R^d} |T_a f|^2 w  \lesssim \int_{\R^d} |J_m f|^2 \mathcal{A}_{\rho,0,R}w \lesssim \int_{\R^d} |f|^2 R^{2m} \Psi_R^{(L)} \ast \mathcal{A}_{\rho,0,R}w \lesssim \int_{\R^d} |f|^2  \CA_{\rho,m,R}w,
\end{align*}
where $\Psi_R^{(L)}(x):= \frac{R^d}{(1+R^2|x|^2)^{L/2}}$ with $L>d$. The penultimate inequality here follows from the elementary inequality
$$
\int_{\R^d} |J_m f|^2 w \lesssim \int_{\R^d} |f|^2 R^{2m} \Psi_R^{(L)} \ast w,
$$
which holds for any $L>d$, any weight $w$, and functions $f$ such that $\supp(\widehat{f}) \subseteq \{\xi \in \R^d: |\xi| \sim R\}$ with $R\geq 1$. The last inequality follows from noting that $\Psi_R^{(L)} \ast \Psi_{R^\rho}^{(N_0)} \lesssim \Psi_{R^\rho}^{(N_0)}$ choosing $L=N_0$; see Lemma \ref{LemmaConv} in Appendix \ref{app:Phi}.

\subsection{Reduction to the symbol classes $S_{0,0}^0$}

The goal now is to prove Proposition \ref{beta0}, that is, the special case of Theorem \ref{TheoremDyadicPiece} for the symbol classes $S^0_{\rho,\rho}$. We shall see that, thanks to an elementary scaling argument, this reduces itself to the following specific case for the symbol class $S_{0,0}^0$.

\begin{proposition}\label{S000}
Let $a \in S_{0,0}^0$. Then
\begin{equation*}
\int_{\R^d} |T_a f|^2 w \lesssim \int_{\R^d} |f|^2 \mathcal{A}w,
\end{equation*}
where $\mathcal{A}w:=\Psi^{(N_0)} \ast \widetilde{w}$, $\widetilde{w}(x):=\sup_{|y-x| \leq 1} w(y)$ and $\Psi^{(N_0)}(x)=\frac{1}{(1+|x|^2)^{N_0/2}}$ with $N_0>d$.
\end{proposition}

To deduce Proposition \ref{beta0} from this, let $\varphi$ be a smooth function such that $\widehat{\varphi}$ equals 1 in $\{ \xi \in \R^d: |\xi| \sim 1\}$ and has compact Fourier support in a slightly enlargement of it, and let $\varphi_R$ be defined by $\widehat{\varphi}_R(\xi):=\widehat{\varphi}(R^{-1}\xi)$. The Fourier support properties of $f$ allows us to write the reproducing formula $\widehat{f}=\widehat{f} \widehat{\varphi}_R$. We may then replace the symbol $a(x,\xi)$ by $a(x,\xi)\widehat{\varphi}_R(\xi)$, which belongs to the class $S^0_{\rho,\rho}$ uniformly in $R$. For ease of notation, we shall denote the product symbol $a(x,\xi)\widehat{\varphi}_R(\xi)$ by $a(x,\xi)$, but assuming that $a(x,\xi)$ is supported in $\{|\xi| \sim R\}$. Let
$$
\widetilde{a}(x,\xi):=a(R^{-\rho}x,R^{\rho} \xi).
$$
It is easy to verify from the differential inequalities \eqref{symbol}, and the support property of $a(x,\xi)$, that the new symbol $\widetilde{a}$ belongs to the class $S_{0,0}^0$ uniformly in $R$; note that $\widetilde{a}$ is $\xi$-supported in an annulus of width $O(R^{1-\rho})$. 
%

The change of variables $x \mapsto R^{-\rho}x$, $\xi \mapsto R^{\rho}\xi$ together with Proposition \ref{S000} lead to
$$
\int_{\R^d}|T_af|^2 w = \int_{\R^d} |T_{\widetilde{a}}f_R|^2 w_R \lesssim \int_{\R^d} |f_R|^2 \mathcal{A}w_R
$$
for functions $f$ such that $\supp(\widehat{f}) \subseteq \{|\xi| \sim R\}$, where 
$$
w_R(x):=w(R^{-\rho}x)R^{-\rho d}
$$
and
$$
\widehat{f}_R(\xi):=\widehat{f}(R^{\rho}\xi)R^{\rho d}.
$$
Proposition \ref{beta0} now follows from noting that
$$
\mathcal{A}w_R(R^{\rho}x)R^{\rho d}= \mathcal{A}_{\rho,0,R}w(x),
$$
which is a consequence of the definitions of $\CA$ and $\CA_{\rho,0, R}$, along with an elementary scaling argument.

We briefly compare Proposition \ref{S000} with Theorem \ref{MainThm} for the specific case of the symbol class $S_{0,0}^0$. In contrast with $\mathscr{M}_{0,0}$, the operator $\CA$ fails to be bounded in any Lebesgue space, making Proposition \ref{S000} not as interesting on its own. Our method to turn it into a bounded maximal operator consists of smoothing out the weight $w$ by a suitable average, in the spirit of \eqref{smoothWeight}. As explained in Section \ref{sec:ProofMain}, this may be done after a first dyadic frequency decomposition of the operator $T_af$. 

\subsection{The symbol class $S_{0,0}^0$: proof of Proposition \ref{S000}}

In this section we assume that $a \in S^0_{0,0}$. We first observe that the weight $w$ is pointwise controlled by $\CA w$. This is contained in the following lemma, which we borrow from \cite{Ben2014}; see \cite{BCSV} for the origins of this. Its short proof is included for completeness.
\begin{lemma}[\cite{Ben2014,BCSV}]
$w \lesssim \mathcal{A}w.$
\end{lemma}

\begin{proof}
It is trivial to observe that $w \leq \widetilde{w}$, so we only need to show $\widetilde{w} \lesssim \CA w$. By translation invariance, it is enough to see that
$$
\widetilde{w}(0) \lesssim \CA w(0).
$$
As $\widetilde{w}\geq 0$ and $\Psi^{(N_0)}(y) \gtrsim 1$ for $|y|\leq 1$,
$$
\CA w(0)=\int_{\R^d} \frac{1}{(1+|y|^2)^{N_0/2}}\widetilde{w}(y)dy \gtrsim \int_{|y|\leq 1} \widetilde{w}(y)dy.
$$
Let $B_1, \dots, B_{2^d}$ be the intersections of the unit ball with the $2^d$ coordinate hyperoctants of $\R^d$. It is enough to show that there exists $\ell^* \in \{1,\dots,2^d\}$ such that $\widetilde{w}(y)\geq \widetilde{w}(0)$ for all $y \in B_{\ell^*}$, as then
$$
\CA \widetilde{w}(0) \gtrsim \int_{|y|\leq 1} \widetilde{w}(y)dy = \int_{B_{\ell^*}} \widetilde{w}(y)dy + \sum_{\ell \neq \ell^*} \int_{B_\ell} \widetilde{w}(y)dy \geq |B_{\ell^*}| \widetilde{w}(0) \gtrsim \widetilde{w}(0),
$$
which would conclude the proof. We prove our claim by contradiction. Suppose that for each $1 \leq \ell \leq 2^d$ there exist $y_\ell \in B_\ell$ such that $\widetilde{w}(y_\ell) < \widetilde{w}(0)$. By the definition of $\widetilde{w}$,
$$
\sup_{|z-y_\ell|\leq 1} w(z) < \widetilde{w}(0) \:\:\:\: \text{ for } \:\:\:\: 1 \leq \ell \leq 2^d.
$$
As
$$
\{|z| \leq 1\} \subseteq \bigcup_{\ell=1}^{2^d} \{|z-y_\ell| \leq 1\},
$$
we have
$$
\widetilde{w}(0)= \sup_{|z|\leq 1} w(z) \leq \sup_{\bigcup_{\ell=1}^{2^d} \{|z-y_\ell| \leq 1\}} w(z) = \max_{1\leq \ell \leq 2^d} \sup_{|z-y_\ell| \leq 1} w(z) < \max_{1 \leq \ell \leq 2^d} \widetilde{w}(0) = \widetilde{w}(0),
$$
which is of course a contradiction.
\end{proof}

The above lemma reduces the proof of Proposition \ref{S000} to the weighted inequality
\begin{equation}\label{S000maxop}
\int_{\R^d} |T_a f|^2 \mathcal{A}w \lesssim \int_{\R^d} |f|^2 \CA w.
\end{equation}
Defining the operator $S f:=T_a ((\mathcal{A}w)^{-1/2}f) (\mathcal{A}w)^{1/2}$, it is enough to show
\begin{equation}\label{L2}
\int_{\R^d} |Sf|^2 \lesssim \int_{\R^d} |f|^2
\end{equation}
with bounds independent of $w$; \eqref{S000maxop} just follows by taking $f=(\CA w)^{1/2} f$ in \eqref{L2}. Observe first that $(\CA w)^{\ell}$ is a well-defined function for any $\ell \in \R$, as $\CA w >0$. Also, the operator $S$ is well-defined for $f \in \mathcal{S}$; this is due to the fact that any power of $\CA w$ has polynomial growth, as well as all its derivatives, see the forthcoming Lemma \ref{PolyGrowth}. Leibniz's formula ensures then that $(\CA w)^\ell f \in \mathcal{S}$ for any $\ell \in \R$, and that $S$ maps $\CS$ to $\CS$.

\begin{lemma}\label{PolyGrowth}
For any $\ell \in \R$ and any $\gamma \in \N^d$, 
$$|D^\gamma (\CA w)^\ell(x)|\lesssim (\CA w)^{\ell}(x) \lesssim (1+|x|^2)^{N_0|\ell|/2} (\CA w)^{\ell}(0).$$
\end{lemma}

\begin{proof}
From the trivial fact that $|D^\gamma \Psi^{(N_0)} (x)| \lesssim \Psi^{(N_0)}(x)$ for any $\gamma \in \N^d$, by definition of $\CA$ we have 
$$|D^\gamma \CA w(x)| \leq |D^\gamma \Psi^{(N_0)}| \ast \widetilde{w} (x) \lesssim \Psi^{(N_0)} \ast \widetilde{w} (x) = \CA w (x),$$
as $\widetilde{w} \geq 0$. The chain rule quickly reveals
$$
|D^\gamma (\CA w)^{\ell}(x)|\lesssim (\CA w)^{\ell}(x).
$$
For the second inequality, by Lemma \ref{HarnackGoodDecay} in Appendix \ref{app:Phi}, one has
$$
\CA w (0)\frac{1}{(1+|x|^2)^{N_0/2}} \lesssim \CA w(x) \lesssim (1+|x|^2)^{N_0/2} \CA w(0).
$$
Then, if $\ell>0$, $(\CA w)^\ell (x) \lesssim (1+|x|^2)^{N_0 \ell/2} \CA w(0)$, and if $\ell<0$, $(\CA w)^{\ell}(x) \lesssim (1+|x|^2)^{N_0 |\ell|/2} (\CA w)^\ell (0)$, which concludes the proof.
\end{proof}
%
%
%

We shall prove the $L^2$-boundedness of the operator $S$ from an application of the Cotlar--Stein principle to a suitable family of operators. To construct such a family we introduce the following partition of unity. Let $\psi$ be a smooth, nonnegative function supported in the unit cube $Q=\{x \in \R^d : |x_j| \leq 1\}$ and such that
\begin{equation}\label{eqspaced}
\sum_{i \in \Z^d} \psi(x-i)=1,
\end{equation}
and let $a_{\textbf{i}}(x,\xi)=a(x,\xi)\psi(x-i)\psi(\xi-i')$, where $\textbf{i}=(i,i')$. Then
$$
a=\sum_{\textbf{i} \in \Z^{2d}} a_{\textbf{i}}.
$$
This gives a decomposition of the space associated to the $\xi$ variable into balls of radius $O(1)$. Note that in the passage of rescaling the symbol class $S_{0,0}^0$ into $S^0_{1-\alpha,1-\alpha}$, this amounts to a decomposition of the dyadic annulus $\{|\xi| \sim R\}$ into $O(R^{\alpha d})$ balls of radius $O(R^{1-\alpha})$; this would correspond to the prototypical example of the $\alpha$-subdyadic decomposition aforementioned in Section \ref{sec:Multiplier}.\footnote{The decomposition given by $\psi$ was used in the proof of the $L^2$-boundedness of the class $S_{0,0}^0$ that one may find in \cite{bigStein} and it is a prior instance of the subdyadic analysis further developed in \cite{BCSV,BH,Ben2014,BB}.}

This decomposition allows us to write the operator $S$ as
$$
Sf=\sum_{\textbf{i} \in \Z^{2d}} S_{\textbf{i}}f,
$$
where $S_{\textbf{i}}f=T_{a_{\textbf{i}}} ((\mathcal{A}w)^{-1/2}f) (\mathcal{A}w)^{1/2}$. We aim to apply Lemma \ref{Cotlar--Stein} to the family of operators $\{S_\bi\}_{\bi \in \Z^{2d}}$. To this end we need to establish
$$
\|S_\bi^* S_\bj\|_{2 \to 2} \lesssim c(\bi-\bj)^2
$$
and
$$
\|S_\bi S_\bj^*\|_{2 \to 2} \lesssim c(\bi-\bj)^2
$$
for a family of constants $\{c(\bi)\}_{\bi \in \Z^{2d}}$ such that
$$
\sum_{\bi \in \Z^{2d}} c(\bi) < \infty.
$$
Observe that $S^*_{\textbf{i}}f=(\mathcal{A}w)^{-1/2}T_{a_{\textbf{i}}}^* ((\mathcal{A}w)^{1/2}f)$, where
$$
T_{a_\bi}^*g(y)=\int_{\R^d} \int_{\R^d} e^{i \xi \cdot (y-z)} \overline{a}_{\bi}(z,\xi)g(z) d\xi dz
$$
is a well-defined operator that maps $\mathcal{S}$ to $\mathcal{S}$. The decomposition of the $x$ variable via \eqref{eqspaced} ensures the kernel of the operator $S_\bi^* S_\bj$ to be well defined; also the symmetric role of the $x$ and $\xi$ variables in $a(x,\xi) \in S_{0,0}^{0}$ suggests such a decomposition in the $x$ variable. 

\subsubsection{The $L^2$-boundedness of $S_\bi^* S_\bj$} The operator $S_\bi^* S_\bj$ may be realised as
\begin{align*}
S_\bi^*(S_\bj f)(x)& =(\CA w)^{-1/2}(x) T_{a_\bi}^*(\CA w \; T_{a_\bj}((\CA w)^{-1/2} f))(x) \\
& = (\CA w)^{-1/2}(x)\int_{\R^d} K_{\bi,\bj}(x,z)f(z) (\CA w)^{-1/2}(z)dz,
\end{align*}
where
$$
K_{\bi,\bj}(x,z):=\int_{\R^d}\int_{\R^d}\int_{\R^d} e^{i \xi \cdot (x-y)} e^{i \eta \cdot (y-z)} \overline{a_\bi}(y,\xi)a_\bj (y,\eta) \CA w(y)dy d\eta d\xi;
$$
observe that this kernel is well-defined by the support properties of $a_\bi$ and $a_\bj$. Note that if $i-j \not \in Q$, then $K_{\bi,\bj}=0$.

Integrating by parts in $K_{\bi,\bj}$, after making use of the identities
$$
(I-\Delta_y)^{N_1}e^{i y \cdot (\eta-\xi)}=(1+|\xi-\eta|^2)^{N_1}e^{i y \cdot (\eta-\xi)},
$$
$$
(I-\Delta_\eta)^{N_2}e^{i \eta \cdot (y-z)}=(1+|y-z|^2)^{N_2}e^{i \eta \cdot (y-z)}
$$
and
$$
(I-\Delta_\xi)^{N_3}e^{i \xi \cdot (x-y)}=(1+|x-y|^2)^{N_3}e^{i \eta \cdot (x-y)},
$$
leads to
\begin{align*}
K_{\bi,\bj}(x,z) = \int_{\R^d}\int_{\R^d}\int_{\R^d} & e^{i \xi \cdot (x-y)} e^{i \eta \cdot (y-z)} \frac{(I-\Delta_\xi)^{N_3}}{(1+|x-y|^2)^{N_3}} \frac{(I-\Delta_\eta)^{N_2}}{(1+|y-z|^2)^{N_2}}  \Big( \frac{(I-\Delta_y)^{N_1}}{(1+|\xi-\eta|^2)^{N_1}} (\overline{a_\bi}(y,\xi)a_\bj(y,\eta)\CA w(y)) \Big) dy d\eta d\xi,
\end{align*}
for any $N_1,N_2,N_3 \geq 0$. Observe that $|D^\gamma \psi (y- k)| \leq \|\psi\|_{C^{|\gamma|}} \chi(y-k)$ for any multi-index $\gamma \in \N^d$, where $\chi$ is the characteristic function of $Q$. This, Lemma \ref{PolyGrowth}, and the differential inequalities satisfied by the symbols $a_\bi, a_\bj \in S^0_{0,0}$, allows us to deduce, after an application of Leibniz's formula,
%
\begin{align}
|K_{\bi,\bj}(x,z)| & \lesssim \int_{\R^d}\int_{\R^d} \frac{\chi(\xi-i')\chi(\eta-j')}{(1+|\xi-\eta|^2)^{N_1}}d\xi d\eta \int_{\R^d} \frac{\CA w(y) \chi(y-i) \chi(y-j)}{(1+|y-z|^2)^{N_2}(1+|y-x|^2)^{N_3}}dy \notag \\
& \lesssim \frac{1}{(1+|i'-j'|^2)^{N_1}} \int_{\R^d} \frac{\CA w(y) \chi(y-i) \chi(y-j)}{(1+|y-z|^2)^{N_2}(1+|y-x|^2)^{N_3}}dy; \label{Kbound}
\end{align}
the implicit constant here depends on finitely many $C^k$ norms of $\psi$. Now we apply Schur's test to the kernel
$$
\widetilde{K_{\bi,\bj}}(x,z)=K_{\bi,\bj}(x,z) (\CA w)^{-1/2}(x) (\CA w)^{-1/2}(z)
$$
with the auxiliary functions $h_1=h_2=(\CA w)^{1/2}$. We check first that the integral condition with respect to $z$ is satisfied. Observe that from Lemma \ref{LemmaConv} in Appendix \ref{app:Phi}, $(\CA w) \ast \Psi^{(N_0)} \lesssim \CA w$. Using this, and taking $2N_2=2N_3=N_0>d$ in \eqref{Kbound}, we have
\begin{align*}
\int_{\R^d} |\widetilde{K_{\bi,\bj}}(x,z)|h_1(z)dz & \lesssim \frac{(\CA w)^{-1/2}(x)}{(1+|i'-j'|^2)^{N_1}} \int_{\R^d} \int_{\R^d} \frac{\CA w(y) \chi(y-i) \chi(y-j)}{(1+|y-z|^2)^{N_2}(1+|y-x|^2)^{N_3}}dzdy \\
& \lesssim \frac{(\CA w)^{-1/2}(x)}{(1+|i'-j'|^2)^{N_1}} \int_{\R^d} \frac{\CA w(y)}{(1+|y-x|^2)^{N_3}}dy \\
& \lesssim \frac{(\CA w)^{1/2}(x)}{(1+|i'-j'|^2)^{N_1}}, \:\:\:\: \text{ if } i-j \in Q,
\end{align*}
for any $N_1 \geq 0$. On the other hand, $\widetilde{K_{\bi,\bj}}=0$ if $i-j \not \in Q$, so combining both cases,
$$
\int_{\R^d} |\widetilde{K_{\bi,\bj}}(x,z)|h_1(z)dz \lesssim \frac{(\CA w)^{1/2}(x)}{(1+|\bi-\bj|^2)^{N_1}},
$$
for any $N_1 \geq 0$. As the integral condition with respect to the $x$ variable is symmetric, Lemma \ref{SchurLemma} yields
\begin{equation}\label{S^*S}
\|S_\bi^* S_\bj\|_{2 \to 2} \lesssim \frac{1}{(1+|\bi-\bj|^2)^{N_1}}
\end{equation}
for any $N_1\geq 0$. The constant $c(\bi)=(1+|\bi|^2)^{-N_1/2}$ will be sufficient for an application of the Cotlar--Stein lemma.

\subsubsection{The $L^2$-boundedness of $S_\bi S_\bj^*$}

Our goal now is to see that $\|S_\bi S_\bj^*\|_{2 \to 2}$ also satisfies the bound \eqref{S^*S}. The operator $S_\bi S_\bj^*$ may be realised as
\begin{align*}
S_\bi(S_\bj^* f)(x)& =(\CA w)^{1/2}(x)T_{a_\bi} ((\CA w)^{-1} T_{a_\bj}^* ((\CA w)^{1/2}f))(x) \\
& = (\CA w)^{1/2}(x) \int_{\R^d} \int_{\R^d}  e^{i \xi \cdot (x-y)} a_\bi (x,\xi) (\CA w)^{-1}(y)T_{a_\bj}^*(f(\CA w)^{1/2})(y)dy d\xi \\
& = (\CA w)^{1/2}(x) \int_{\R^d} L_{\bi,\bj}(x,z) f(z) (\CA w)^{1/2}(z) dz,
\end{align*}
where $L_{\bi,\bj}$ is taken to be the formal sum
\begin{equation}\label{Lijdef}
L_{\bi,\bj}(x,z):= \sum_{k \in \Z^d} L_{\bi,\bj}^k (x,z)
\end{equation}
and
$$
L^k_{\bi,\bj}(x,z):=\int_{\R^d}\int_{\R^d}\int_{\R^d}e^{i \xi \cdot (x-y)} e^{i \eta \cdot (y-z)} a_\bi (x,\xi) \overline{a_\bj}(z,\eta) (\CA w)^{-1}(y)\psi(y-k) dy d\xi d\eta.
$$
Observe that, a priori, the formal sum
$$
L_{\bi,\bj}(x,z)= \sum_{k \in \Z^d} L_{\bi,\bj}^k (x,z)= \int_{\R^d}\int_{\R^d}\int_{\R^d}e^{i \xi \cdot (x-y)} e^{i \eta \cdot (y-z)} a_\bi (x,\xi) \overline{a_\bj}(z,\eta) (\CA w)^{-1}(y)dy d\xi d\eta,
$$
may not be well-defined, as the triple integral in the right hand side does not necessarily converge absolutely. For this reason, we introduce the partition of unity \eqref{eqspaced} in the $y$ variable; the integral that defines $L_{\bi,\bj}^k$ is now absolutely convergent. Our analysis below shows, in particular, that such a sum is finite.
%

Again, integration by parts with respect to $y,\eta,\xi$ gives
\begin{align*}
L_{\bi,\bj}(x,z)&= \sum_{k \in \Z^d} \int_{\R^d}\int_{\R^d}\int_{\R^d}e^{i \xi \cdot (x-y)} e^{i \eta \cdot (y-z)} \frac{(I-\Delta_\xi)^{N_3}}{(1+|x-y|^2)^{N_3}} \frac{(I-\Delta_\eta)^{N_2}}{(1+|y-z|^2)^{N_2}} \Big( \frac{a_\bi (x,\xi) \overline{a_\bj}(z,\eta)}{(1+|\eta-\xi|^2)^{N_1}}\Big) \times \\
& \hspace{3cm} \times (I-\Delta_y)^{N_1} \big((\CA w)^{-1}(y)\psi(y-k)\big) dy d\xi d\eta,
\end{align*}
for any $N_1,N_2,N_3 \geq 0$, and the same observations as in the previous case allows us to deduce, after an application of Leibniz's formula, $|L_{\bi,\bj}(x,z)|$ is bounded by
\begin{align*}
\sum_{k \in \Z^d} \int_{\R^d}\int_{\R^d}\int_{\R^d} \frac{\chi(x-i)\chi(\xi-i')}{(1+|x-y|^2)^{N_3}} \frac{\chi(z-j)\chi(\eta-j')}{(1+|y-z|^2)^{N_2}} \frac{(\CA w)^{-1}(y) \chi(y-k)}{(1+|\eta-\xi|^2)^{N_1}}  dy d\xi d\eta.
\end{align*}
As the functions $\{\chi(\cdot-k)\}_{k \in \Z^d}$ have bounded overlap, we may sum in the $k$ variable and
\begin{align}
|L_{\bi,\bj}(x,z)| & \lesssim \int_{\R^d}\int_{\R^d} \frac{\chi(\xi-i')\chi(\eta-j')}{(1+|\eta-\xi|^2)^{N_1}} d\xi d\eta  \int_{\R^d}  \frac{(\CA w)^{-1}(y)}{(1+|y-z|^2)^{N_2}} \frac{ \chi(z-j) \chi(x-i)}{(1+|x-y|^2)^{N_3}} dy \notag \\
& \lesssim \frac{1}{(1+|i'-j'|^2)^{N_1}} \int_{\R^d}  \frac{(\CA w)^{-1}(y)}{(1+|y-z|^2)^{N_2}} \frac{ \chi(z-j) \chi(x-i)}{(1+|x-y|^2)^{N_3}} dy. \label{Lbound}
\end{align}
The integration in the $y$ variable is finite, so the sum taken in the definition of $L_{\bi,\bj}$ in \eqref{Lijdef} is well defined. In particular, for $N_2=N_3>N_0+d$, it is possible to show that
\begin{equation}\label{ConditionSchur}
\int_{\R^d} \int_{\R^d} \frac{(\CA w)^{-1}(y)}{(1+|y-z|^2)^{N_2}} \frac{ \chi(z-j) \chi(x-i)}{(1+|x-y|^2)^{N_2}} dz dy \lesssim  \frac{(\CA w)^{-1} (x)}{(1+|i-j|^2)^{N_2/2}}.
\end{equation}
As the role of the variables $x$ and $z$ is symmetric here, the same follows with $(\CA w)^{-1}(x)$ replaced by $(\CA w)^{-1}(z)$ in the right hand side of \eqref{ConditionSchur}. 

Assuming the estimate \eqref{ConditionSchur} is true, one may successfully apply Schur's test to the kernel
$$
\widetilde{L}_{\bi,\bj}(x,z)=L_{\bi,\bj}(x,z) (\CA w)^{1/2} (x) (\CA w)^{1/2} (z)
$$
with auxiliary functions $h_1=h_2=(\CA w)^{-1/2}$. Using \eqref{Lbound} and \eqref{ConditionSchur}, we have
\begin{align*}
\int_{\R^d} |\widetilde{L}_{\bi,\bj}(x,z)|h_1(z)dz& \lesssim  \frac{(\CA w)^{1/2} (x)}{(1+|i'-j'|^2)^{N_1}} \int_{\R^d} \int_{\R^d}  \frac{(\CA w)^{-1}(y)}{(1+|y-z|^2)^{N_2}} \frac{ \chi(z-j) \chi(x-i)}{(1+|x-y|^2)^{N_2}} dy dz \\
& \lesssim \frac{(\CA w)^{1/2} (x)}{(1+|i'-j'|^2)^{N_1}} \frac{(\CA w)^{-1}(x)}{(1+|i-j|)^{N_2/2}} \\
& \lesssim \frac{(\CA w)^{-1/2}(x)}{(1+|\bi-\bj|^2)^{N_2/2}},
\end{align*}
for $N_2>N_0+d$; the last inequality follows from taking $N_1=N_2/2$. As the integral condition with respect to the $x$ variable is symmetric, an application of Lemma \ref{SchurLemma} yields
$$
\|S_\bi S_\bj^*\|_{2 \to 2} \lesssim \frac{1}{(1+|\bi-\bj|^2)^{N_2/2}}
$$
for any $N_2>N_0+d$. 

\subsubsection{The $L^2$-boundedness of $S$}

We just saw that the family of operators $\{S_\bi\}_{i \in \Z^{2d}}$ satisfies the bounds
\begin{equation}\label{boundi*j}
\|S_\bi^* S_\bj\|_{2 \to 2} \lesssim \frac{1}{(1+|\bi-\bj|^2)^{N_1}},
\end{equation}
for any $N_1 \geq 0$, and
$$
\|S_\bi S_\bj^*\|_{2 \to 2} \lesssim \frac{1}{(1+|\bi-\bj|^2)^{N_2/2}}
$$
for any $N_2>N_0+d$. Taking $N_1=N_2/2$ in \eqref{boundi*j} and noting that the series
$$
\sum_{\bi \in \Z^{2d}} \frac{1}{(1+|\bi|^2)^{N_2/4}} < \infty
$$
for $N_2/2 > 2d$, an application of the Cotlar--Stein almost orthogonality principle (Lemma \ref{Cotlar--Stein}) to the family of operators $\{S_\bi\}_{i \in \Z^{2d}}$ ensures that
$$
\|S\|_{2 \to 2} \leq \sum_{\bi \in \Z^{2d}} \frac{1}{(1+|\bi|^2)^{N_2/4}} < \infty,
$$
provided $N_2 > \max\{N_0+d, 4d\}$. As we may choose $N_2$ as large as we please, the estimate \eqref{L2} follows. This finishes the proof of Theorem \ref{MainThm}, provided the estimate \eqref{ConditionSchur} is shown to be true.

\subsubsection{The validity of the estimate \eqref{ConditionSchur}}

At this stage we are only left with proving \eqref{ConditionSchur}, that is
\begin{equation*}
\int_{\R^d} \int_{\R^d} \frac{(\CA w)^{-1}(y)}{(1+|y-z|^2)^{N_2}} \frac{ \chi(z-j) \chi(x-i)}{(1+|x-y|^2)^{N_2}} dz dy \lesssim  \frac{(\CA w)^{-1} (x)}{(1+|i-j|^2)^{N_2/2}}.
\end{equation*}
To this end, we divide the range for the $y$-integration into two half-spaces, $H_x$ and $H_z$, that contain the points $x$ and $z$ respectively and that are the result of splitting $\R^d$ by a hyperplane perpendicular to the line segment joining $x$ and $z$ at its midpoint. Note that for $y \in H_x$, $|y-z| \geq \frac{1}{2} |x-z|$, so
$$
\frac{1}{(1+|y-z|^2)^{N_2}} \leq \frac{2^{2N_2}}{(1+|x-z|^2)^{N_2}}
$$
and
\begin{align*}
\int_{\R^d} \int_{H_x} \frac{(\CA w)^{-1}(y)}{(1+|y-z|^2)^{N_2}} \frac{ \chi(z-j) \chi(x-i)}{(1+|x-y|^2)^{N_2}} dydz & \lesssim \int_{\R^d} \frac{\chi(z-j) \chi(x-i)}{(1+|x-z|^2)^{N_2}} dz \int_{H_x} \frac{(\CA w)^{-1}(y)}{(1+|x-y|^2)^{N_2}}dy \\
& \lesssim \frac{1}{(1+|i-j|^2)^{N_2/2}} \int_{\R^d} \frac{(\CA w)^{-1}(y)}{(1+|x-y|^2)^{N_2/2}}dy.
\end{align*}
Similarly, for $y \in H_z$, $|x-y| \geq \frac{1}{2} |x-z|$, so
$$
\frac{1}{(1+|x-y|^2)^{N_2}} \leq \frac{2^{2N_2}}{(1+|x-z|^2)^{N_2}}
$$
and
\begin{align*}
\int_{\R^d} \int_{H_z} \frac{(\CA w)^{-1}(y)}{(1+|y-z|^2)^{N_2}} \frac{ \chi(z-j) \chi(x-i)}{(1+|x-y|^2)^{N_2}} dydz & \lesssim \int_{\R^d} \int_{H_z} \frac{(\CA w)^{-1}(y)}{(1+|y-z|^2)^{N_2}} \frac{ \chi(z-j) \chi(x-i)}{(1+|x-z|^2)^{N_2}} dydz.
\end{align*}
By the elementary inequality
$$
\frac{1}{(1+|y-z|^2)^{N_2/2}} \frac{1}{(1+|x-z|^2)^{N_2/2}} \lesssim \frac{1}{(1+|x-y|^2)^{N_2/2}},
$$
which is a simple consequence of the triangle inequality, we have
\begin{align*}
\int_{\R^d} \int_{H_z} \frac{(\CA w)^{-1}(y)}{(1+|y-z|^2)^{N_2}} \frac{ \chi(z-j) \chi(x-i)}{(1+|x-z|^2)^{N_2}} dydz & \lesssim \int_{\R^d} \frac{ \chi(z-j) \chi(x-i)}{(1+|x-z|^2)^{N_2/2}} dz \int_{H_z} \frac{(\CA w)^{-1}(y)}{(1+|x-y|^2)^{N_2/2}}  dy \\
& \lesssim \frac{1}{(1+|i-j|^2)^{N_2/2}} \int_{\R^d} \frac{(\CA w)^{-1}(y)}{(1+|x-y|^2)^{N_2/2}}dy. 
\end{align*}
Putting both estimates together,
$$
\int_{\R^d} \int_{\R^d} \frac{(\CA w)^{-1}(y)}{(1+|y-z|^2)^{N_2}} \frac{ \chi(z-j) \chi(x-i)}{(1+|x-y|^2)^{N_2}} dz dy \lesssim \frac{(\CA w)^{-1} \ast \Psi^{(N_2)} (x)}{(1+|i-j|^2)^{N_2/2}},
$$
so the inequality \eqref{ConditionSchur} is satisfied if
$$
(\CA w)^{-1} \ast \Psi^{(N_2)}(x) \lesssim (\CA w)^{-1} (x).
$$
As $\widetilde{w} \geq 0$, by Lemma \ref{HarnackGoodDecay},
$$
\Psi^{(N_0)} \ast \widetilde{w} (x) \geq \frac{1}{(1+|x-y|^2)^{N_0/2}} \Psi^{(N_0)} \ast \widetilde{w} (y),
$$
so by definition of $\CA w$,
$$
(\CA w)^{-1}(x) \leq(1+|x-y|^2)^{N_0/2} (\CA w)^{-1} (y);
$$
in particular
$$
(\CA w)^{-1}(x-y) \leq (1+|y|^2)^{N_0/2} (\CA w)^{-1} (x).
$$
Thus
\begin{align*}
(\CA w)^{-1} \ast \Psi^{(N_2)} (x) &= \sum_{l \in \Z^d} \int_{l +[0,1]^d} (\CA w)^{-1}(x-y)\Psi^{(N_2)}(y)dy \\
&  \leq \sum_{l \in \Z^d} (\CA w)^{-1}(x) \int_{l+[0,1]^d} (1+|y|^2)^{(N_0-N_2)/2}dy \\
& \lesssim (\CA w)^{-1}(x) \sum_{l \in \Z^d} (1+|l|^2)^{(N_0-N_2)/2} \\
&\lesssim (\CA w)^{-1} (x), 
\end{align*}
provided $N_2>N_0+d$, and the inequality \eqref{ConditionSchur} follows.


\appendix

\section{Properties of $\Psi^{(N)}$}\label{app:Phi}

Here we briefly recall some elementary properties of the function $\Psi^{(N)}_R(x):=\frac{R^{d}}{(1+R^2|x|^2)^{N/2}}$ to which we appealed to in our proof of Theorem \ref{MainThm}.

\begin{lemma}\label{LemmaConv}
Let $N>d$. Let $R \geq K$ denote two different scales. Then $\Psi_R^{(N)} \ast \Psi_K^{(N)} \lesssim \Psi_K^{(N)}.$
\end{lemma}

\begin{proof}
We need to show
$$
\int_{\R^d} \frac{R^d}{(1+R^2|y-x|^2)^{N/2}} \frac{K^d}{(1+K^2|y|^2)^{N/2}}dy \lesssim \frac{K^d}{(1+K^2|x|^2)^{N/2}}
$$
for any $x \in \R^d$. Observe first that if $K|x| \leq 1$, the estimate is trivial, as
$$
\frac{K^d}{(1+K^2|y|^2)^{N/2}} \leq K^d \leq \frac{2^{N/2} K^d}{(1+K^2|x|^2)^{N/2}}
$$
and the integral
$$
\int_{\R^d} \frac{R^d}{(1+R^2|y-x|^2)^{N/2}} dy < \infty
$$
provided $N>d$.

If $K|x| \geq 1$, we divide $\R^d$ into two half-spaces $H_x$ and $H_0$, that contain the points $x$ and $0$ respectively and that are the result of splitting $\R^d$ by a hyperplane perpendicular to the line segment joining $x$ and the origin $0$ at its midpoint. If $y \in H_x$, then $|y|\geq |x|/2$ and
$$
\frac{K^d}{(1+K^2|y|^2)^{N/2}} \leq \frac{2^N K^d}{(1+K^2|x|^2)^{N/2}}.
$$
Thus
\begin{align*}
\int_{H_x}  \frac{R^d}{(1+R^2|y-x|^2)^{N/2}} \frac{K^d}{(1+K^2|y|^2)^{N/2}}dy & \leq \frac{2^N K^d}{(1+K^2|x|^2)^{N/2}} \int_{H_x} \frac{R^d}{(1+R^2|y-x|^2)^{N/2}} dy \\
&\lesssim \frac{K^d}{(1+K^2|x|^2)^{N/2}}.
\end{align*}
If $y \in H_0$, we have $|y-x| \geq |x|/2$. Similarly,
$$
\frac{R^d}{(1+R^2|y-x|^2)^{N/2}} \leq \frac{2^N R^d }{(1+R^2|x|^2)^{N/2}} \leq \frac{2^N R^d}{R^N |x|^N} = \frac{2^N R^{d-N}}{|x|^N}.
$$
As $R>K$, $N>d$ and $K|x| \geq 1$,
$$
\frac{2^N R^{d-N}}{|x|^N} \leq \frac{2^N K^{d-N}}{|x|^N} = \frac{2^N 2^{N/2} K^d}{(2K^2|x|^2)^{N/2}} \lesssim \frac{K^d}{(1+K^2|x|^2)^{N/2}}, 
$$
and arguing as in the previous case, this concludes the proof.
\end{proof}


For the case $R=1$, we simply denote $\Psi^{(N)}(x):=\frac{1}{(1+|x|^2)^{N/2}}$. We have the following Harnack-type property.

\begin{lemma}\label{HarnackGoodDecay}
For $w \geq 0$,
$$
w \ast \Psi^{(N)} (x) \gtrsim \frac{1}{(1+|x-y|^2)^{N/2}} w \ast \Psi^{(N)} (y).
$$
\end{lemma}

\begin{proof}
The triangle inequality quickly reveals that $(1+|x|^2)^{-N/2} \gtrsim (1+|x-y|^2)^{-N/2} (1+|y|^2)^{-N/2}$ for any $N \geq 0$. Then, as $w \geq 0$,
$$
\frac{w(z)}{(1+|x-z|^2)^{N/2}} \gtrsim \frac{1}{(1+|x-y|^2)^{N/2}} \frac{w(z)}{(1+|y-z|^2)^{N/2}},
$$
just by replacing $x \mapsto x-z$, $y\mapsto y-z$. The result follows from integrating with respect to the $z$ variable.
\end{proof}

\section{Symbolic calculus}\label{app:SymbCalc}

This appendix is devoted to providing a proof of Theorem \ref{RefinedSymbolicCalculus}, which is a very specific quantitative version of the symbolic calculus in Hörmander \cite{Hormander}. As is mentioned in Section \ref{subsec:symbcalc}, the order of the error symbol $e^N \in S^{m-N(1-\delta)+d\delta+\kappa \delta+\epsilon}_{\rho,\delta}$ is not necessarily sharp here, but it naturally arises from our proof. Nevertheless, such an order is admissible for our purposes, as one may choose $N$ large enough so that $e^N$ is of sufficiently large negative order. Our proof follows the same structure as that given in Stein \cite{bigStein} for the standard symbol classes $S^m$.

To justify our computations, we technically should replace $a$ by $a_\varepsilon$, where $a_\varepsilon(x,\xi)=a(x,\xi)\psi(\varepsilon x, \varepsilon \xi)$ and $\psi \in C_0^\infty (\R^d \times \R^d)$ with $\psi(0,0)=1$. The symbol $a_\varepsilon$, which has compact support, satisfies the same differential inequalities as $a$ uniformly in $0 < \varepsilon \leq 1$. As our estimates will be independent of $\varepsilon$, the passage to the limit when $\varepsilon \to 0$ gives the desired result; we refer to \cite{bigStein} for these standard details. Such considerations allow us to suppress the dependence on $\varepsilon$ in what follows.

Observe that we may write
$$
T_{\widehat{\varphi}_R}(T_af)(x)=\int_{\R^d}\int_{\R^d} c(x,\xi)e^{i(x-z)\cdot \xi}f(z)dz d\xi,
$$
where 
$$
c(x,\xi)=\int_{\R^d} \int_{\R^d} \widehat{\varphi}_R(\eta)a(y,\xi)e^{i(x-y)\cdot (\eta-\xi)}dy d\eta=\int_{\R^d}\widehat{\varphi}_R(\xi+\eta)\widehat{a}(\eta,\xi)e^{i x \cdot \eta}d\eta,
$$
and $\: \widehat{a} \:$ denotes Fourier transform with respect to the $x$ variable. We first obtain an estimate depending on the size of the support of $a$; such dependence will be later removed in the second part of the proof.

\subsection{Assuming $a(x,\xi)$ has compact support in the $x$-variable}

Integrating by parts,
$$
\widehat{a}(\eta,\xi)=\int_{\R^d} \frac{e^{i x \cdot \eta}}{(1+|\eta|^2)^M} (I-\Delta_x)^{M}a(x,\xi)dx,
$$
so
\begin{equation}\label{Decaybhat}
|\widehat{a}(\eta,\xi)| \lesssim (1+|\eta|)^{-2M}(1+|\xi|)^{m+2M\delta},
\end{equation}
for any $M \geq 0$; the implicit constant above depends on the size of the support of $a$ in the $x$ variable. For $\widehat{\varphi}_R(\xi+\eta)$ we use Taylor's formula around the point $\xi$,
$$
\widehat{\varphi}_R(\xi+\eta)=\sum_{|\gamma|< N} \frac{1}{\gamma!} \partial_\xi^\gamma \widehat{\varphi}_R(\xi) \eta^\gamma + \mathfrak{R}_N(\xi,\eta),
$$
where $\mathfrak{R}_N$ is the remainder in Taylor's theorem and is bounded by
$$
|\mathfrak{R}_N(\xi,\eta)|\lesssim \max_{|\gamma|=N} \max_{\zeta} |\partial_\xi^\gamma \widehat{\varphi}_R(\zeta)| |\eta|^N,
$$
where the maximum in $\zeta$ is taken on the line segment joining $\xi$ to $\xi+\eta$. Thus
\begin{align*}
c(x,\xi) & = \sum_{|\gamma|<N} \frac{1}{\gamma!} \int_{\R^d} \partial_\xi^\gamma \widehat{\varphi}_R(\xi)\eta^\gamma \widehat{a}(\eta,\xi)e^{ix \cdot \eta}d\eta + \int_{\R^d} \mathfrak{R}_N(\xi,\eta)\widehat{a}(\eta,\xi)e^{ix\cdot \eta}d\eta \\
& = \sum_{|\gamma|<N} \frac{i^{-|\gamma|}}{\gamma!} \partial_\xi^\gamma \widehat{\varphi}_R(\xi) \partial_x^\gamma a(x,\xi) + \int_{\R^d} \mathfrak{R}_N(\xi,\eta)\widehat{a}(\eta,\xi)e^{ix\cdot \eta}d\eta
\end{align*}
and
$$
e^N(x,\xi)=\int_{\R^d} \mathfrak{R}_N(\xi,\eta)\widehat{a}(\eta,\xi)e^{ix\cdot \eta}d\eta.
$$
We need to show that the $e^N \in S^{m-N(1-\delta)+d\delta+\kappa \delta+\epsilon}_{\rho,\delta}$ and satisfies the differential inequalities \eqref{ErrorDiffIneq}.

Observe that, for $\gamma$ such that $|\gamma|=N$,
$$
\partial_\xi^\gamma \widehat{\varphi}_R(\zeta)=R^{-N} (\partial_\xi^\gamma \widehat{\varphi})(R^{-1}\zeta) \lesssim R^{-\epsilon} (1+|\zeta|)^{-N+\epsilon},
$$
as the support condition on $\widehat{\varphi}$ ensures $|\zeta|\sim R \sim |\zeta|+1$, since $R>1$. This leads to the following estimates for the remainder,
$$
|\mathfrak{R}_N(\xi,\eta)| \lesssim R^{-\epsilon} |\eta|^N (1+|\xi|)^{-N+\epsilon} \:\:\: \text{for} \:\:\: |\xi|\geq 2 |\eta|,
$$
and
$$
|\mathfrak{R}_N(\xi,\eta)| \lesssim R^{-\epsilon} |\eta|^N \:\:\: \text{for} \:\:\: |\xi|\leq 2|\eta|,
$$
as $N \geq \epsilon$. Using the estimate \eqref{Decaybhat} in the form
$$
|\widehat{a}(\eta,\xi)|\lesssim (1+|\eta|)^{-2M_1}(1+|\xi|)^{m+2M_1 \delta} \:\:\: \text{for} \:\:\: |\xi|\geq 2 |\eta|,
$$
and
$$
|\widehat{a}(\eta,\xi)|\lesssim (1+|\eta|)^{-2M_2}(1+|\xi|)^{m+2M_2 \delta} \:\:\: \text{for} \:\:\: |\xi|\leq 2|\eta|,
$$
where $M_1, M_2 \geq 0$, we have
\begin{align*}
|e^N(x,\xi)| & \lesssim R^{-\epsilon} (1+|\xi|)^{m+2M_1\delta-N+\epsilon} \int_{|\xi|\geq 2 |\eta|} (1+|\eta|)^{-2M_1} |\eta|^N d\eta \\
& \;\;\;\;\;\;\; + R^{-\epsilon} (1+|\xi|)^{m+2M_2\delta} \int_{|\xi| \leq 2|\eta|} (1+|\eta|)^{-2M_2} |\eta|^N d\eta \\
& \lesssim  R^{-\epsilon} (1+|\xi|)^{m+2M_1\delta-N+\epsilon} + R^{-\epsilon} (1+|\xi|)^{m+2M_2\delta - 2M_2+N+d}
\end{align*}
provided $-2M_1+N+d <0$ and $-2M_2+N+d<0$. Choosing
$$
M_1=(N+d+\kappa)/2
$$
and
$$
M_2=\frac{2N+d(1-\delta)-\kappa \delta-\epsilon-N\delta}{2(1-\delta)},
$$
which clearly satisfies the condition $-2M_2+N+d<0$, as $N>\epsilon+\kappa \delta$, one has
$$
|e^N(x,\xi)| \lesssim R^{-\epsilon} (1+|\xi|)^{m-N(1-\delta)+d\delta+\kappa \delta+\epsilon}.
$$
In view of the definitions of the symbols $c$ and $e^N$, the use of the Leibniz formula and the condition $\rho \leq 1$ allows one, by the same arguments as above, to deduce the differential inequalities \eqref{ErrorDiffIneq} for all multi-indices $\nu, \sigma \in \N^d$.

\subsection{The case of general $a(x,\xi)$}

Observe that it suffices to prove the differential inequalities \eqref{ErrorDiffIneq} for $x$ near an arbitrary but fixed point $x_0$; in particular we prove them for $x$ such that $|x-x_0|\leq 1/2$, with bounds independent of $x_0$. To this end, let $\theta$ be a smooth function which equals $1$ on $|y-x_0|\leq 1$ and supported in $|y-x_0|\leq 2$, and write $a=\theta a +(1-\theta)a=a_1+a_2$. For $a_1$, one may argue as before and write
$$
c_1(x,\xi)=\sum_{|\gamma| < N} \frac{i^{-|\gamma|}}{\gamma !} (\partial_\xi^\gamma \widehat{\varphi}_R(\xi))(\partial_x^\gamma a_1 (x,\xi)) + \int_{\R^d} \mathfrak{R}_N(\xi,\eta)\widehat{a}_1(\eta,\xi)e^{i x \cdot \eta} d\eta .
$$
As $a_1=a$ for $|x-x_0| \leq 1/2$ and the size of the support of $a_1$ in the $x$ variable is constant and independent of $x_0$, the previous argument reveals that the symbol
$$
e^N_1(x,\xi):=c_1(x,\xi)-\sum_{|\gamma| < N} \frac{i^{-|\gamma|}}{\gamma !} (\partial_\xi^\gamma \widehat{\varphi}_R(\xi))(\partial_x^\gamma a (x,\xi))
$$
satisfies the differential inequalities \eqref{ErrorDiffIneq} for $|x-x_0| \leq 1/2$, with bounds independent of $x_0$. As
$$
|e^N (x,\xi)| \leq |e^N_1(x,\xi)| + |c_2(x,\xi)|,
$$
where $c_2$ is the symbol of $T_{\widehat{\varphi}_R} \circ T_{a_2}$, it is enough to show that $c_2$ satisfies the same estimates as $e^N_1$. Indeed, we will show that for $|x-x_0|\leq 1/2$,
$$
|c_2(x,\xi)| \lesssim R^{-\epsilon} (1+|\xi|)^{m-\bar N}
$$
for any $\bar N \geq 0$; the proof then follows by taking 
$$
\bar N = N(1-\delta)-d\delta-\kappa \delta-\epsilon,
$$
which is nonnegative for $N > \frac{d\delta+\kappa \delta+\epsilon}{1-\delta}$.

Recall that
$$
c_2(x,\xi)=\int_{\R^d} \int_{\R^d} \widehat{\varphi}_R(\eta)a_2(y,\xi)e^{i(x-y)\cdot (\eta-\xi)}dy d\eta.
$$
Integrating by parts with respect to the $\eta$ variable,
$$
c_2(x,\xi)=\int_{\R^d} \int_{\R^d} \frac{\Delta_{\eta}^{N_1} \widehat{\varphi}_R(\eta)}{|x-y|^{2N_1}}a_2(y,\xi)e^{i(x-y)\cdot (\eta-\xi)}dy d\eta,
$$
which is a convergent integral, as for $|x-x_0|\leq 1/2$ and $|y-x_0|\geq 1$, we have $|x-y|\geq 1/2$. Integrating by parts with respect to the $y$ variable,
$$
c_2(x,\xi)=\int_{\R^d} \int_{\R^d} \frac{\Delta_{\eta}^{N_1} \widehat{\varphi}_R(\eta)}{(1+|\eta-\xi|^2)^{N_2}}(I-\Delta_y)^{N_2}\Big(\frac{a_2(y,\xi)}{|x-y|^{2N_1}}\Big)e^{i(x-y)\cdot (\eta-\xi)}dy d\eta.
$$
In view of the differential inequalities satisfied by $\widehat{\varphi}_R$ and $a_2$,
$$
|c_2(x,\xi)|\lesssim R^{-\epsilon} \int_{\R^d} \int_{\R^d} \frac{(1+|\eta|)^{-2N_1+\epsilon}(1+|\xi|)^{m+2N_2\delta}}{(1+|\eta-\xi|)^{2N_2}(1+|x-y|)^{2N_1}}dyd\eta.
$$
The integration in $y$ is finite if we choose $N_1>d/2$. The triangle inequality trivially reveals
$$
\frac{1}{(1+|\eta-\xi|)^{2N_2}} \leq \frac{(1+|\eta|)^{2N_2}}{(1+|\xi|)^{2N_2}},
$$
for any $N_2\geq 0$, so
\begin{align*}
|c_2(x,\xi)| \lesssim R^{-\epsilon} (1+|\xi|)^{m-2N_2(1-\delta)} \int_{\R^d} (1+|\eta|)^{-2N_1+2N_2+\epsilon}d\eta \lesssim R^{-\epsilon} (1+|\xi|)^{m-\bar N},
\end{align*}
provided we take $2N_2(1-\delta)=\bar N$ and $N_1$ satisfying $2N_1-2N_2-\epsilon>d$, that is
$$
N_1 > \frac{d+\epsilon+\bar N/(1-\delta)}{2}.
$$
Observe that any such $N_1$ also satisfies the required condition $N_1>d/2$, as $\bar N \geq 0$.

In view of the definition of $c_2$, the use of the Leibniz formula and of similar arguments to the ones exposed above leads one to deduce that
$$
|\partial_x^\nu \partial_\xi^\sigma c_2(x,\xi)| \lesssim R^{-\epsilon} (1+|\xi|)^{m-\rho|\sigma|+\delta|\nu| - \bar N}
$$
for any $\bar N \geq 0$ and all multi-indices $\nu, \sigma \in \N^d$, so we may conclude that $e^N$ satisfies the required differential inequalities \eqref{ErrorDiffIneq}.

\bibliographystyle{abbrv} 

\bibliography{RevisedBeltranSymbols}

%



%
%
%
%
%

\end{document}